\documentclass[11pt,reqno]{amsart}
\usepackage{amsfonts,amsmath,amssymb,amsthm,eucal}
\usepackage{enumitem}
\usepackage{hyperref}
\usepackage[noabbrev,capitalize]{cleveref}

\DeclareFontFamily{U}{mathx}{}
\DeclareFontShape{U}{mathx}{m}{n}{<-> mathx10}{}
\DeclareSymbolFont{mathx}{U}{mathx}{m}{n}
\DeclareMathAccent{\widehat}{0}{mathx}{"70}
\DeclareMathAccent{\widecheck}{0}{mathx}{"71}

\allowdisplaybreaks 

\newtheorem{theorem}{Theorem}[section]
\newtheorem{proposition}[theorem]{Proposition}
\newtheorem{lemma}[theorem]{Lemma}
\theoremstyle{definition}
\newtheorem{notation}[theorem]{Notation}
\newtheorem{remark}[theorem]{Remark}
\newtheorem{step}{Step}

\newcommand*{\TheoremAnchorPrefix}{theorem}
\def\theHstep{\TheoremAnchorPrefix.\the\value{step}}

\DeclareMathOperator{\Gal}{Gal}
\DeclareMathOperator{\Span}{span}
\DeclareMathOperator{\Tr}{Tr}
\DeclareMathOperator{\pr}{pr}

\newcommand{\F}{{\mathbb F}}
\newcommand{\N}{{\mathbb N}}
\newcommand{\R}{{\mathbb R}}
\newcommand{\Q}{{\mathbb Q}}
\newcommand{\Z}{{\mathbb Z}}
\newcommand{\Qz}{{\Q(\zeta)}}
\newcommand{\Zz}{{\Z[\zeta]}}
\newcommand{\Ku}{K^\times}
\newcommand{\Kuk}{(K^\times)^k}
\newcommand{\Lu}{L^\times}
\newcommand{\Fp}{\F_p}
\newcommand{\Fpu}{\Fp^\times}
\newcommand{\weil}{W^{K,s}}
\newcommand{\cW}{{\mathcal W}}
\newcommand{\ws}{\cW_{K,s}}
\newcommand{\gr}{L[\Ku]}
\newcommand{\mchars}{\widehat{\Ku}}
\newcommand{\Fpgen}{\gamma}
\newcommand{\ferm}[1]{\Vert{#1}\Vert_s}
\newcommand{\ferms}[1]{\Vert{#1}\Vert_s^s}

\newcommand{\conj}[1]{\overline{{#1}}}
\newcommand{\sums}[1]{\sum_{\substack{#1}}}
\newcommand{\ggen}[1]{\left\langle{#1}\right\rangle}

\renewcommand{\epsilon}{\varepsilon}
\renewcommand{\phi}{\varphi}

\title{Rationality of Four-Valued Families of Weil Sums of Binomials}
\author{Daniel J.\ Katz}
\author{Allison E.\ Wong}
\thanks{Daniel J.\ Katz is with the Department of Mathematics, California State University, Northridge, USA.  Allison E.\ Wong was with the Department of Mathematics, California State University, Northridge and is with the Department of Mathematics, University of California, Davis.  This paper is based upon work of both authors supported in part by the National Science Foundation under Grant CCF-1815487, and upon work of Daniel J.\ Katz supported in part by the National Science Foundation under Grant CCF-2206454.}
\date{06 April 2024}

\begin{document}

\begin{abstract}
We investigate the rationality of Weil sums of binomials of the form $W^{K,s}_u=\sum_{x \in K} \psi(x^s - u x)$, where $K$ is a finite field whose canonical additive character is $\psi$, and where $u$ is an element of $K^{\times}$ and $s$ is a positive integer relatively prime to $|K^\times|$, so that $x \mapsto x^s$ is a permutation of $K$.
The Weil spectrum for $K$ and $s$, which is the family of values $W^{K,s}_u$ as $u$ runs through $K^\times$, is of interest in arithmetic geometry and in several information-theoretic applications.
The Weil spectrum always contains at least three distinct values if $s$ is nondegenerate (i.e., if $s$ is not a power of $p$  modulo $|K^\times|$, where $p$ is the characteristic of $K$).
It is already known that if the Weil spectrum contains precisely three distinct values, then they must all be rational integers.
We show that if the Weil spectrum contains precisely four distinct values, then they must all be rational integers, with the sole exception of the case where $|K|=5$ and $s \equiv 3 \pmod{4}$.
\end{abstract}

\maketitle

\section{Introduction} \label{INTRODUCTION}

In this paper, we assume that \(K\) is a finite field of characteristic \(p\) and order \(q = p^n\).
Let \(\zeta = \exp(2 \pi i / p)\). 
The canonical additive character of \(K\) is \(\psi \colon K \to \Qz\) given by \(\psi(x) = \zeta^{\Tr(x)}\), where \(\Tr \colon K \to \Fp\) with \(\Tr(x) = x + x^p + \cdots + x^{q/p}\).
We use \(s\) to denote an {\it invertible exponent over $K$}, that is, a positive integer with \(\gcd(s,q-1) = 1\).
This ensures that \(s\) has a multiplicative inverse, \(1/s\), modulo \(q-1\) and  makes \(x \mapsto x^s\) a permutation of the field \(K\) with inverse map \(x \mapsto x^{1/s}\).
For each \(u \in K\), we define
\begin{equation}\label{Weil}
\weil_u = \sum_{x \in K} \psi(x^s - u x) = \sum_{x \in K} \zeta^{\Tr(x^s - u x)},
\end{equation}
which is a Weil sum of a binomial (if \(u \neq 0\)) or a Weil sum of a monomial (if \(u=0\)).
When the field \(K\) and the exponent \(s\) are clear from context, we omit the superscript and write \(W_u\).
Note that Weil sum values lie in \(\Zz\), the ring of algebraic integers in \(\Qz\).
In fact, it is known that they lie in \(\Zz \cap \R\); see \cite[Theorem 2.1(c)]{Katz-2012}, or see \cite[Theorem 2.3]{Trachtenberg} for an earlier equivalent statement in terms of crosscorrelation of linear recursive sequences.
\begin{theorem}[Trachtenberg, 1970]\label{Real} 
If \(K\) is a finite field and \(s\) is an invertible exponent over $K$, then \(\weil_u \in \R\) for every \(u \in K\).
\end{theorem}

A {\it multiset} of elements from a set $X$ is a function $\mu$ from $X$ into the nonnegative integers, where for $x \in X$ the value $\mu(x)$ is the {\it frequency} (number of instances) of $x$ in the multiset.
Thus, $\mu$ represents a normal set if and only if it maps $X$ into $\{0,1\}$ (in which case $\mu$ is identified with the subset $\mu^{-1}(\{1\})$ of $X$).
The {\it Weil spectrum for the field \(K\) and the exponent \(s\)} is the multiset of values \(\weil_u\) as \(u\) runs through \(\Ku\).
That is, the Weil spectrum is a multiset of elements from $\Zz$, where a given value $A \in \Zz$ has a frequency, written \(N^{K,s}_A\) (or \(N_A\) when \(K\) and \(s\) are clear from context), with
\begin{equation}\label{Nora}
N^{K,s}_A = |\{u \in \Ku : \weil_u = A\}|.
\end{equation}
We define the {\it value set for the field \(K\) and the exponent \(s\)}, written $\ws$, to be the set of distinct values in the Weil spectrum, that is,
\begin{equation}\label{William}
\ws = \{\weil_u: u \in \Ku\}.
\end{equation}  
Note that we do not record \(\weil_0\) in \(\ws\), but this value is always \(0\) because \(x \mapsto x^s\) is a permutation of \(K\) and \(\sum_{x \in K} \psi(x) = 0\).

The evaluation and estimation of Weil sums has been studied extensively \cite{Kloosterman, Davenport-Heilbronn, Akulinichev, Karatsuba, Carlitz-1978, Carlitz-1979, Coulter, Cochrane-Pinner-2003, Cochrane-Pinner-2011, Shparlinski-Voloch}, including special cases such as Kloosterman sums, which are of the form $W^{K,|K|-2}_u-1$.
Weil sums are used to count points in algebraic sets over finite fields; see, for example, Sections 7.7 and  7.11 of \cite{Katz-2019} and \cref{ALGEBRAS} of this paper.
In the Kloosterman case, the Weil spectra for fields of characteristic $2$ and $3$ were studied in \cite{Lachaud-Wolfmann} and \cite{Katz-Livne}, and Sections 7.2--7.4 of \cite{Katz-2019} describe applications of Weil spectra in information theory, which we summarize here.
The Walsh spectrum of the permutation $x\mapsto x^s$ of $K$ is obtained from the Weil spectrum by also including the value $\weil_0=0$.
The Walsh spectrum measures the nonlinearity of the permutation, which indicates its resistance to linear cryptanalysis.
The crosscorrelation spectrum of two maximum length linear recursive sequences is obtained by subtracting $1$ from each value in the Weil spectrum.
This crosscorrelation spectrum determines the performance of communications networks and remote sensing systems employing these sequences for modulation.
Weil spectra also determine the weight distribution of certain error correcting codes, thus indicating the performance of the codes.

For a finite field $K$, we say that two exponents $s$ and $s'$ are {\it equivalent} to mean that \(s' \equiv p^k s^{\ell} \pmod{q-1}\) for some \(k \in \Z\) and \(\ell \in \{-1,1\}\); this defines an equivalence relation, and equivalent exponents produce the same Weil spectrum by \cite[Theorems 2.4, 2.5]{Trachtenberg} (in the language of crosscorrelation), or see \cite[Lemmas 7.5.2, 7.5.6]{Katz-2019}\footnote{Lemma 7.5.6 of \cite{Katz-2019} has a typographical error: \(a^{1/d}\) should be fixed to read \(a^{-1/d}\) there.}.
We say that \(s\) is \textit{degenerate over \(K\)} to mean that it is equivalent to \(1\), that is, $s$ is a power of $p$ modulo $q-1$.
If \(K\) has four or fewer elements, then all exponents are degenerate over \(K\); larger finite fields always have at least one nondegenerate exponent (see \cite[Lemma 7.5.4]{Katz-2019}).
If \(s\) is degenerate, then \(\ws = \{0,q\}\) if \(q > 2\) and \(\ws = \{q\}\) if \(q=2\); see \cite[Corollary 7.5.5]{Katz-2019}.

We say the Weil spectrum for \(K\) and \(s\) is \textit{\(v\)-valued} (resp., \textit{at least $v$-valued}, \textit{at most $v$-valued}) to mean that \(|\ws| = v\) (resp., \(|\ws|\geq v\), \(|\ws| \leq v\)).
Thus, Weil spectra of degenerate exponents are at most \(2\)-valued, and Helleseth showed that Weil spectra of nondegenerate exponents are always at least \(3\)-valued in \cite[Theorem 4.1]{Helleseth}.
\begin{theorem}[Helleseth, 1976]\label{Tor}
Let $K$ be a finite field and $s$ be an invertible exponent over $K$.
Then the Weil spectrum for $K$ and $s$ is at least $3$-valued if and only if \(s\) is nondegenerate over $K$.
\end{theorem}
There is much interest in which pairs \((K,s)\) produce Weil spectra with few values (e.g., \(3\)-valued or \(4\)-valued spectra).
All known $3$-valued spectra have been classified into ten infinite families (see \cite[Table 7.1]{Katz-2019}), and \(4\)-valued spectra have been studied in \cite[Theorems 3-6, 3-7]{Niho}, \cite[Theorem 4.13]{Helleseth}, \cite[Proposition 1]{Dobbertin}, \cite[Theorem 6]{Helleseth-Rosendahl}, \cite[Theorem 23]{Dobbertin-Felke-Helleseth-Rosendahl}, \cite[Theorem II.5]{Zhang-Li-Feng-Ge}, and \cite[Theorem 1]{Xia-Helleseth-Wu}.
Although each Weil sum value is always an algebraic integer in some cyclotomic extension of \(\Q\), one observes that Weil spectra with few distinct values often have all of their values in \(\Z\).
We say that the Weil spectrum for \(K\) and \(s\) is \textit{rational} (or that \(\ws\) is {\it rational}) to mean \(\ws \subseteq \Z\).
Helleseth proved a simple criterion for rationality in \cite[Theorem 4.2]{Helleseth}.
\begin{theorem}[Helleseth, 1976]\label{carrot}
Let $K$ be a finite field of characteristic $p$ and $s$ be an invertible exponent over $K$.
Then the Weil spectrum for $K$ and $s$ is rational if and only if \(s \equiv 1 \pmod{p-1}\).
\end{theorem}
Later, in \cite[Theorem 1.7]{Katz-2012}, it was proved that \(3\)-valued Weil spectra are invariably rational.
\begin{theorem}[Katz, 2012]\label{Natalie}
Let $K$ be a finite field and $s$ be an invertible exponent over $K$.
If the Weil spectrum for $K$ and $s$ is $3$-valued, then it is rational.
\end{theorem}
Thus, in view of \cref{carrot}, when $K$ is a field of characteristic $p$ and $s\not\equiv 1 \pmod{p-1}$, the Weil spectrum for $K$ and $s$ cannot be $3$-valued.
Katz and Langevin set an open problem \cite[Problem 3.6]{Katz-Langevin}, part of which is to find an analogue of \cref{Natalie} for \(4\)-valued spectra.
The main result of this paper is this analogue, which we now state.
\begin{theorem} \label{MainTheorem}
Let $K$ be a finite field and $s$ be an invertible exponent over $K$.
If the Weil spectrum for $K$ and $s$ is $4$-valued, then it is rational unless \(K = \F_5\) and \(s \equiv 3 \pmod{4}\) (in which case \(\ws = \{(5 \pm \sqrt{5})/2, \pm \sqrt{5}\}\)).
\end{theorem}
By \cref{carrot}, this means that, other than in the exceptional case when $|K|=5$ and $s\equiv 3 \pmod{4}$, the condition $s\equiv 1 \pmod{p-1}$ is necessary for the Weil spectrum to be \(4\)-valued.
Since the Walsh spectrum of the power permutation $x\mapsto x^s$ over $K$ is obtained from the Weil spectrum for $K$ and $s$ by including $\weil_0=0$, Theorems \ref{Natalie} and \ref{MainTheorem} show that all the values in a four-valued Walsh spectrum must lie in $\Z$.

The remainder of this paper is devoted to proving \cref{MainTheorem}.
We start in \cref{NUMBERS} by using Galois theory and algebraic number theory to study the structure of Weil spectra. 
Then, in \cref{BOUNDS}, we present some archimedean and $p$-adic bounds on Weil sum values.
\cref{SETS} introduces some algebraic sets over finite fields, which we then relate to Weil sums in \cref{ALGEBRAS} via a group algebra.
Finally, we prove \cref{MainTheorem} in \cref{ACTIONS}.

\section{Algebraic number theory}\label{NUMBERS}

In this section we introduce the number systems that are used in our proof of \cref{MainTheorem}.  
Algebraic number theory provides several results that constrain the structure of Weil spectra and thus help us achieve our proof.

Recall that \(K\) is a finite field of characteristic \(p\) and order \(q = p^n\), that \(s\) is a positive integer such that \(\gcd(s,q-1) = 1\), and that \(\zeta = \exp(2 \pi i/p)\).
We use \(\N\) to denote the set of nonnegative integers and $\Z_+$ to denote the set of strictly positive integers.
We know that \(\Gal(\Qz/\Q)\) is a cyclic group of order \(p-1\); an element of this Galois group fixes all elements of \(\Q\) and maps \(\zeta\) to \(\zeta^j\) for some \(j \in \Fpu\).
Let \(\Fpgen\) denote a primitive element of the prime subfield \(\Fp\) and let \(\sigma\) denote the automorphism in \(\Gal(\Qz/\Q)\) that maps \(\zeta\) to \(\zeta^\Fpgen\): note that \(\sigma\) is a generator of the Galois group.
Then \cite[Theorem 2.1(b)]{Katz-2012} shows that \(\sigma(W_u) = W_{\Fpgen^{1-1/s} u}\) for every \(u \in K\), where \(1/s\) is interpreted as the multiplicative inverse of \(s\) modulo \(p-1\).
Thus, \(\sigma\) maps the value set \(\ws\) (see \eqref{William}) to itself.
From now on, we let \(\tau\colon \ws \to \ws\) be the permutation obtained by restricting \(\sigma\), so that for every $u \in K$, we have
\begin{equation}\label{Rachel}
\tau(W_u) = W_{\Fpgen^{1-1/s} u},
\end{equation}
where \(1/s\) is interpreted as the multiplicative inverse of \(s\) modulo \(p-1\).

The following result indicates important relationships between the exponent $s$, the characteristic $p$ of the field $K$, the order of $\tau$, the order of the element $\gamma^{1-1/s}$ in \eqref{Rachel}, and the degree of the extension of $\Q$ generated by the values in the Weil spectrum.
\begin{proposition}\label{Simeon}
The following are all equal:
\begin{enumerate}[label = (\roman*)]
\item the order of the permutation $\tau$ of $\ws$,
\item the degree, $[\Q(\ws):\Q]$, of the field extension of the rationals generated by $\ws$,
\item the order of $\Fpgen^{1-1/s}$ in $\Fpu$ (where $1/s$ indicates the multiplicative inverse of $s$ modulo $p-1$), and
\item the quantity $(p-1)/\gcd(p-1,s-1)$.
\end{enumerate}
Let $m$ denote the common value of these.
If \(p=2\), then \(m=1\), but if \(p > 2\), then \(p \equiv 1 \pmod{2 m}\).
\end{proposition}
\begin{proof}
Since $\Q(\ws)$ is a subfield of $\Qz$ and since $\Gal(\Qz/\Q)$ is a cyclic group generated by $\sigma$, the Galois correspondence shows that $[\Q(\ws):\Q]$ equals the order of the restriction to $\Q(\ws)$ of $\sigma$, which is the same as the order of $\tau$.
Lemma 5.3 of \cite{Aubry-Katz-Langevin} shows that $[\Q(\ws):\Q]$ equals $(p-1)/\gcd(p-1,s-1)$, which is the order of $\Fpgen^{1-1/s}=(\Fpgen^{1/s})^{s-1}$ because $\Fpgen$ has order $p-1$ and $s$ is invertible modulo $p-1$ (since $\gcd(s,q-1)=1$).

If \(p=2\), then \(\Qz = \Q(-1) = \Q\), so \(\Q(\ws) = \Q\) and \(m=1\).
When \(p > 2\), \cref{Real} shows that \(\Q(\ws)\) is a subfield of \(\Qz \cap \R = \Q(\zeta+\zeta^{-1})\), an extension of $\Q$ of degree $(p-1)/2$, and so $m \mid (p-1)/2$. 
\end{proof}

\begin{remark}
\cref{Simeon} shows that \(\ws\) is rational when \(p=2\) or \(3\).
\end{remark}  
Recall from \eqref{Nora} that the frequency of a value $A$ in the Weil spectrum is \(N_A = |\{u \in \Ku : W_u = A\}|\).
The action of $\tau$ on the Weil spectrum gives us information about these frequencies.
\begin{lemma}\label{Judah}
Suppose that \(\tau\) has order \(m\), and let \(A_0, A_1, \ldots, A_{k-1}\) be distinct elements of \(\ws\) that \(\tau\) permutes in a \(k\)-cycle, that is, \(\tau(A_i)=A_{i+1}\) for every \(i\in\Z/k\Z\).
Then \(k \mid m\) and \(N_{A_0} = N_{A_1} = \cdots = N_{A_{k-1}}\), which is a multiple of \(m/k\).
\end{lemma}
\begin{proof}
Let \(U_i = \{u \in \Ku: W_u = A_i\}\) for each \(i \in \Z/k\Z\) and let \(\lambda = \Fpgen^{1-1/s}\), where we interpret \(1/s\) as the multiplicative inverse of \(s\) modulo \(p-1\).
For \(u \in U_0\) and \(j \in \Z\) we have, by \eqref{Rachel}, that \(W_{\lambda^j u} = \tau^j(W_u) = \tau^j(A_0) = A_{j \bmod{k}}\).
In particular, we have \(k \mid m\) since \(\tau\) has order \(m\). 
Moreover, \(W_{\lambda^k u} = A_0 = W_u\), so \(U_0\) is a union of cosets of the subgroup \(\langle\lambda^k\rangle\) of the group \(\Ku\).
Since \(\lambda\) is of order \(m\) by \cref{Simeon} and \(k\mid m\), this subgroup is of order \(m/k\), and so \(N_{A_0} = |U_0|\) is a multiple of \(m/k\).
Lastly, for any \(j \in \Z\), the map \(u \mapsto \lambda^j u\) provides a bijection from \(U_0\) to \(U_{j \bmod{k}}\) because we have seen that \(W_{\lambda^j u} = A_{j \bmod{k}}\) for every \(u \in U_0\), and we can similarly prove \(W_{\lambda^{-j} v} = \tau^{-j}(W_v) = \tau^{-j}(A_{j \bmod{k}}) = A_0\) for every \(v \in U_{j \bmod{k}}\).
\end{proof}

Let \(f\) be a permutation of a finite set $X$.
The \textit{cycle type of \(f\)} is the multiset of lengths of cycles that is obtained when \(f\) is written as a composition of disjoint cycles.
Note that the sum of the values in the cycle type of a permutation \(f\) is equal to the size of the set being permuted.
We say that $f$ {\it is a single cycle} to mean that $f$ can be written as a single cycle that contains all elements of $X$.
The next two results explore constraints on the cycle type of $\tau$.
\begin{lemma}\label{Dan}
When \(p=2\), the cycle type of \(\tau\) is a collection of \(|\ws|\) instances of \(1\).
When \(p\) is odd, the cycle type of \(\tau\) contains no number larger than \((p-1)/2\).
\end{lemma}
\begin{proof}
Let \(m\) be the order of \(\tau\).
When \(p=2\), the field \(\Qz = \Q(-1) = \Q\), so \(\sigma\) and \(\tau\) are identity maps.
When \(p\) is odd, \cref{Simeon} implies that \(m \leq (p-1)/2\), so the desired result follows  since \(m\) is the least common multiple of all the numbers in the cycle type of \(\tau\).
\end{proof}

\begin{proposition}\label{Naphtali}
The permutation \(\tau\) is a single cycle if and only if \(K = \F_2\) (and then \(s\) is degenerate and \(\tau\) is a \(1\)-cycle). 
\end{proposition}
\begin{proof}
Suppose \(p=2\).
\cref{Dan} shows that \(\tau\) is a single cycle if and only if \(|\ws| = 1\), which happens exactly when \(K = \F_2\) (and then every exponent is degenerate and $\tau$ is a $1$-cycle).

Now suppose \(p\) is odd. 
Let \(\ws = \{A_0, \ldots, A_{k-1}\}\) and suppose for a contradiction that \(\tau\) is a single cycle.  
Then \(N_{A_0} = \cdots = N_{A_{k-1}}\) by \cref{Judah}.
The sum \(\sum_{u \in \Ku} W_u\) of \(q-1\) Weil sum values is equal to \(q\) by \cite[Proposition 3.1(b)]{Katz-2012}, so that
\begin{align*}
k N_{A_0} 				   	& = q-1 \text{ and} \\
N_{A_0}(A_0 + \cdots + A_{k-1}) 	& =  q.
\end{align*}
Note that \(A_0+\cdots+A_{k-1} \in \Z\) since it is an algebraic integer fixed by \(\sigma\) (of which $\tau$ is a restriction).
Thus, \(N_{A_0}\) is a common divisor of \(q\) and \(q-1\), and hence \(N_{A_0} = 1\) and \(k=q-1\).
But then, by \cref{Dan}, we must have \((p-1)/2 \geq k = q-1 \geq p-1\), which is impossible.  
\end{proof}

\section{Bounds on Weil sum values} \label{BOUNDS}

In this section we discuss some archimedean and non-archimedean bounds on the Weil sum \(\weil_u\) that are used in proving the main result (\cref{MainTheorem}).
Recall that we use \(\N\) to refer to the set of nonnegative integers.
We use the \(p\)-adic valuation, \(v_p\).
One begins with \(v_p \colon \Z \to \N \cup \{\infty\}\), where \(v_p(0) = \infty\) and \(v_p(a) = \max\{j \in \N: p^j \mid a\}\) when \(a \not= 0\).
Then one extends the domain of \(v_p\) to \(\Q\) by letting \(v_p(a/b) = v_p(a) - v_p(b)\) when \(a,b \in \Z\) and \(b \not= 0\).
Furthermore, one can extend the domain of \(v_p\) to \(\Qz\), in which case \(v_p(\zeta-1) = 1/(p-1)\); see \cite[Theorem 4.1]{Lang-Algebra}, \cite[p.~7]{Lang}, and \cite[p.~218]{Helleseth}. 
For the purposes of this paper, the most important facts about \(v_p\) (which we shall use without proof) are that \(v_p(a b) = v_p(a) + v_p(b)\) and that \(v_p(a+b) \geq \min\{v_p(a), v_p(b)\}\), with \(v_p(a+b) = \min\{v_p(a), v_p(b)\}\) if \(v_p(a) \neq v_p(b)\).

From \eqref{Weil}, we know that Weil sums are sums of $p$th roots of unity, so we first explore linear combinations of these roots.
\begin{lemma}\label{Gad}
For any \(t \in \Q\) and any \(v \in \Qz\), there is one and only one way to write \(v\) as a \(\Q\)-linear combination of \(1,\zeta,\ldots,\zeta^{p-1}\) such that the coefficients sum to \(t\).
\end{lemma}
\begin{proof}
Our claim will follow if we show that the map $\phi\colon \Q^p \to \Q(\zeta)\times\Q$ with $\phi(w_0,w_1,\ldots,w_{p-1})=(w_0+w_1\zeta+\cdots+w_{p-1}\zeta^{p-1},w_0+w_1+\cdots+w_{p-1})$ is an isomorphism of $\Q$-vector spaces. Since $\phi$ is clearly a $\Q$-linear map between two $\Q$-vector spaces of dimension $p$, it suffices to show that $\ker(\phi)$ is trivial.  Let $\pr_1\colon\Q(\zeta)\times\Q\to\Q(\zeta)$ and $\pr_2\colon\Q(\zeta)\times\Q\to\Q$ be the projection maps.  Since $\{1,\zeta,\ldots,\zeta^{p-1}\}$ spans the $(p-1)$-dimensional $\Q$-space $\Q(\zeta)$ and has dependence relation $1+\zeta+\cdots+\zeta^{p-1}=0$, we know that $\pr_1\circ\phi$ is surjective, which makes $\ker(\pr_1\circ\phi)$ equal to the $1$-dimensional space $\Span_\Q\{(1,1,\ldots,1)\}$.  Then $\ker(\phi)$ is a subspace of $\ker(\pr_1\circ\phi)$, but $(\pr_2\circ\phi)(1,1,\ldots,1)\not=0$, so $\ker(\phi)$ must be trivial.
\end{proof}

Now we shall apply the previous result to obtain an archimedean bound on nondegenerate Weil sums. 
Recall that we let \(K\) be a finite field with characteristic \(p\) and order \(q = p^n\) and that \(s\) is a positive integer with \(\gcd(s,q-1) = 1\).
\begin{lemma}\label{Asher}
For any \(u \in K\), there exist unique \(w_0, \ldots, w_{p-1} \in \N\) with \(w_0 > 0\) such that \(\sum_{i=0}^{p-1} w_i = q\) and \(W_u = \sum_{i=0}^{p-1} w_i \zeta^i\).
If \(s\) is nondegenerate, then \(w_i < q\) for every \(i \in \{0,1, \ldots, p-1\}\) and \(|W_u| < q\).
\end{lemma}
\begin{proof}
By definition \eqref{Weil}, a Weil sum \(W_u\) is a sum of \(q\) terms from the set \(\{\zeta^0, \zeta^1, \ldots, \zeta^{p-1}\}\), so we can write \(W_u = \sum_{i=0}^{p-1} w_i \zeta^i\) for some \(w_0, \ldots, w_{p-1} \in \N\) such that \(\sum_{i=0}^{p-1} w_i = q\). 
The uniqueness of this representation follows from \cref{Gad}. 
Note that \(w_0 > 0\) because one term in \(W_u\) is \(\psi(0^s - u \cdot 0) = \zeta^0\).

When \(s\) is nondegenerate, \cite[Theorem 2.1(f)]{Katz-2012} tells us \(|W_u| < q\), which makes it impossible for \(w_i = q\) for any \(i\) (else \(|W_u| = |q \zeta^i| = q\)).
\end{proof}
The next two results explore $p$-adic bounds on Weil sums.
\begin{lemma}\label{Issachar}
For all \(u \in K\), we have \(v_p(W_u) > 0\). 
\end{lemma}
\begin{proof}
This is \cite[Theorem 2.1(e)]{Katz-2012}.  For an equivalent version in terms of crosscorrelation, see \cite[Theorem 4.5]{Helleseth}.
\end{proof}

\begin{lemma}\label{Zebulun}
Suppose that \(s\) is nondegenerate.  
If \(u \in K\) and \(v_p(W_u) \geq v_p(q) = n\), then \(W_u = 0\). 
In particular, either \(v_p(W_u) < v_p(q)\) or else \(v_p(W_u) = \infty\).
\end{lemma}
\begin{proof}
Let \(u \in K\) and use \cref{Asher} to write $W_u = \sum_{0 \leq i < p} w_i \zeta^i$, where $w_0,\ldots,w_{p-1}$ are nonnegative integers that are strictly less than $q$ with $\sum_{0 \leq i < p} w_i=q$. 
Suppose that \(v_p(W_u) \geq v_p(q) = n\).
Then \(q = p^n\) divides \(W_u\) in \(\Zz\), so that $W_u=q r$ for some $r \in \Z[\zeta]$, which we write as $\sum_{0 \leq i < p} r_i \zeta^i$, where each $r_i=w_i/q$ is a nonnegative rational number strictly less than $1$ with $\sum_{0 \leq i < p} r_i=1$.
Then we write $r$ as $\sum_{0 \leq i < p-1} (r_i-r_{p-1}) \zeta^i$, which is the unique $\Q$-linear combination of $1,\zeta,\ldots,\zeta^{p-2}$ equal to $r$, and since $r \in \Z[\zeta]$, the coefficients $r_i-r_{p-1}$ are all in $\Z$.
Since $0 \leq r_i < 1$ for every $i$, this forces $r_0=\cdots=r_{p-1}$, so that $r=0$, and then $W_u=0$.
\end{proof}

Recall from \cref{NUMBERS} that \(\Fpgen\) is a primitive element of the prime subfield \(\Fp\) and \(\sigma\) is the generator of \(\Gal(\Qz/\Q)\) that maps \(\zeta\) to \(\zeta^\Fpgen\). 
If \(p\equiv 1 \pmod{4}\), then it is well known from algebraic number theory that \(\Qz \supseteq \Q(\sqrt{p})\) and that \(\Qz/\Q(\sqrt{p})\) is an extension of degree \((p-1)/2\) with Galois group \(\ggen{\sigma^2}\).
The algebraic integers in \(\Q(\sqrt{p})\) are precisely elements of the form \((a + b\sqrt{p})/2\) with \(a,b \in \Z\) and \(a\equiv b \pmod{2}\).
We are interested in how one obtains such elements from Weil sums. 
To explore this, we use Gauss's determination of the quadratic Gauss sum when \(p\equiv 1 \pmod{4}\) (see \cite[Theorem 5.15]{Lidl-Niederreiter}):
\begin{equation}\label{Natasha}
\sum_{i\in\Fpu} \eta(i) \zeta^i=\sqrt{p},
\end{equation}
where \(\eta\) is the quadratic character (Legendre symbol) of \(\Fpu\).

\begin{lemma}\label{Joseph}
Suppose that \(p \equiv 1 \pmod{4}\).
An expression of the form \(\sum_{i\in\Fp} w_i \zeta^i\) with rational coefficients \(w_i\) lies in \(\Q(\sqrt{p})\) if and only if, for every \(i,j \in \Fp\), we have \(w_i = w_j\) when \(\eta(i) = \eta(j)\).
In this case, if we write \(w_+\) for the common value of the \(w_i\)'s with \(\eta(i) = +1\) and \(w_-\) for the common value of the \(w_i\)'s with \(\eta(i) = -1\), then our sum becomes
\[
\left(w_0 - \frac{w_+ + w_-}{2}\right) + \left(\frac{w_+ - w_-}{2}\right) \sqrt{p}.
\]
\end{lemma}
\begin{proof}
Since \(\Gal(\Qz/\Q(\sqrt{p})) = \ggen{\sigma^2}\), we know that \(A=\sum_{i\in\Fp} w_i \zeta^i \in \Q(\sqrt{p})\) if and only if it is fixed by $\sigma^2$, that is, if and only if
\[
\sum_{i \in \Fp} w_i \zeta^i = \sum_{i \in \Fp} w_i \zeta^{i \Fpgen^2} = \sum_{i\in\Fp} w_{\Fpgen^{-2} i} \zeta^i,
\]
and then \cref{Gad} tells us that this happens if and only if $w_i=w_{\Fpgen^{-2} i}$ for every $i\in \Fp$, which is true if and only if $w_i=w_j$ whenever $j \in i \ggen{\Fpgen^2}$, i.e., whenever $\eta(i)=\eta(j)$.
In this case, write \(w_+\) and \(w_-\) as in the statement of this lemma, and then our sum becomes
\begin{align*}
A
& = w_0 + w_+ \sum_{i\in\ggen{\Fpgen^2}} \zeta^i + w_- \sum_{i\in\Fpu\smallsetminus\ggen{\Fpgen^2}} \zeta^i \\
& = w_0 + \left(\frac{w_++w_-}{2}\right) \sum_{i\in\Fpu} \zeta^i + \left(\frac{w_+-w_-}{2}\right) \sum_{i\in\Fpu} \eta(i) \zeta^i, 
\end{align*}
where the penultimate summation is clearly $-1$ and the ultimate one is the quadratic Gauss sum \eqref{Natasha}.
\end{proof}
We now apply the previous result to Weil sums.
\begin{lemma}\label{Benjamin}
Let \(p\) be a prime with \(p\equiv 1 \pmod{4}\) and suppose that \(s\) is an  invertible exponent over \(K\).
Any Weil sum \(\weil_u\) in \(\Q(\sqrt{p})\) can be written uniquely in the form \((I+J \sqrt{p})/2\), where \(I,J\in\Z\).  Furthermore, \(I \equiv J \pmod{2}\) and \(v_p(I) \geq 1\). 
If \(s\) is nondegenerate, then \(-q < -2 (q-1)/(p-1) < I < 2 q\) and \(|J| \leq 2 (q-1)/(p-1) < q\). 
\end{lemma}
\begin{proof}
From Lemmas \ref{Asher} and \ref{Joseph}, it follows that any Weil sum in \(Q(\sqrt{p})\) can be written as 
\[
\left(\frac{2 w_0 - (w_+ + w_-)}{2}\right) + \left(\frac{w_+ - w_-}{2}\right) \sqrt{p},
\]
where \(w_0, w_+, w_- \in \Z\).  Thus, if we let \(I = 2 w_0 - (w_+ + w_-)\) and \(J = w_+ - w_-\), then \(I,J \in \Z\) and our Weil sum is $(I+J\sqrt{p})/2$; since $\{1,\sqrt{p}\}$ is $\Q$-linearly independent, the $I$ and $J$ are uniquely determined.
Since $J \in \Z$, we know that \(v_p(J\sqrt{p})\) has strictly positive $p$-adic valuation, as does the entire Weil sum (by \cref{Issachar}), and so \(v_p(I)\) must be a strictly positive integer.
Note also that \(I \equiv J \pmod{2}\) since, as we stated in the paragraph before \cref{Joseph}, algebraic integers in \(\Q(\sqrt{p})\) are of the form \((a+b \sqrt{p})/2\) where \(a, b \in \Z\) and \(a \equiv b \pmod{2}\).

From now on, let us suppose that \(s\) is nondegenerate.
Then by \cref{Asher}, we know that \(w_0, w_+, w_-\) are all nonnegative integers that are strictly less than \(q\) with \(w_0 \geq 1\) and \(w_0 + (w_+ + w_-)(p-1)/2=q\).
Thus, 
\[
|J| \leq w_+ + w_- \leq 2 \left(\frac{q-1}{p-1}\right),
\]
and since \(w_0 < q\), we know that \(w_++w_- > 0\), so
\[
-2 \left(\frac{q-1}{p-1}\right) < 2 - 2 \left(\frac{q-1}{p-1}\right) \leq I < 2 w_0 + (w_+ + w_-)(p-1) = 2 q,
\]
where since \(p-1 > 2\), we have \(2(q-1)/(p-1) < q-1 < q\).

\end{proof}

\section{Algebraic sets over finite fields} \label{SETS}

In this section, we study a certain type of algebraic set over the finite field \(K\). 
It turns out that these sets are closely related to sums of products of Weil sum values (as we shall see in \cref{ALGEBRAS}), and thus will help us prove our main result (\cref{MainTheorem}).

Recall that \(K\) is a finite field of characteristic \(p\) and order \(q = p^n\) and that \(s\) is a positive integer such that \(\gcd(s,q-1) =  1\).
First, we introduce two notations that enable us to express our algebraic sets very compactly.
\begin{notation}\label{Oganesson}
If $k \in \Z_+$ and $u=(u_1,\ldots,u_k), v=(v_1,\ldots,v_k) \in K^k$ then $u \cdot v$ denotes $u_1 v_1+\cdots+u_k v_k$ and $\ferm{u}$ denotes $(u_1^s+\cdots+u_k^s)^{1/s}$, so that $\ferms{u}=u_1^s+\cdots+u_k^s$.
\end{notation}

\begin{notation}\label{Tennessine}
For \(k \in \Z_+\), \(t=(t_1, \ldots, t_k) \in \Kuk\), and \(a,b\in K\), we use \(Q^t_{a,b}\) to denote the number of solutions \(v = (v_1,\ldots, v_k) \in K^k\) to the system of equations
\begin{align*}
t\cdot v & = a \\
\ferm{v} & = b.
\end{align*}
\end{notation}
The next four results relate various values of $Q^t_{a,b}$ with each other.
\begin{lemma}\label{Livermorium}
For any \(k \in \Z_+\), any \(t \in \Kuk\), and any \(b \in K\) we have  \(\sum_{a \in K} Q^t_{a,b} = \sum_{a \in K} Q^t_{b,a} = q^{k-1}\).
\end{lemma}
\begin{proof}
The second summation counts the points in the hyperplane $t \cdot v = b$ in $K^k$, while the first sum counts points with $\ferms{v} = b^s$, which has the same cardinality because $x\mapsto x^s$ is a permutation of $K$.
\end{proof}

\begin{lemma}\label{Moscovium}
For \(k \in \Z_+\), any \(u \in \Ku\), any \(t \in \Kuk\), and any \(a, b \in K\), we have  \(Q^{u t}_{a,b} = Q^t_{a/u,b}\) and \(Q^t_{u a, u b} = Q^t_{a,b}\).
\end{lemma}
\begin{proof}
The first equality follows from observing that $u t \cdot v = a$ if and only if $t \cdot v = a/u$.
The second follows from the bijection \((v_1, \ldots, v_k) \mapsto (v_1/u, \ldots, v_k/u)\) from the set of points counted by $Q^t_{u a, u b}$ to that counted by $Q^t_{a,b}$.
\end{proof}

\begin{lemma}\label{Flerovium}
Let \(k \in \Z_+\) and \(t \in \Kuk\).
For any \(a \in \Ku\), we have
\begin{equation}\label{chair}
Q^t_{a,0} = Q^t_{0,a} = \frac{q^{k-1} - Q^t_{0,0}}{q-1}.
\end{equation}
Moreover, if \(b \in K\), we have 
\begin{equation}\label{table}
\sum_{a \in \Ku} Q^t_{a,b} = \sum_{a \in \Ku} Q^t_{b,a} 
= \begin{cases}
q^{k-1} - Q^t_{0,0} & \text{ if } b=0 \\[5pt]
\frac{q^k - 2 q^{k-1} + Q^t_{0,0}}{q-1} & \text{ if } b \neq 0.
\end{cases}
\end{equation}
\end{lemma}
\begin{proof}
\cref{Moscovium} shows that $Q^t_{a,0}$ (resp., $Q^t_{0,a}$) has the same value for every $a \in \Ku$, so \eqref{chair} and the $b=0$ case of \eqref{table} follow from \cref{Livermorium}.  
The $b\not=0$ case of \eqref{table} then similarly follows from \cref{Livermorium}, using \eqref{chair}.
\end{proof}
\begin{lemma}\label{Nihonium}
For any \(k \in \Z_+\), any \(b, t_1, \ldots, t_k \in \Ku\), and any \(a \in K\), we have 
\[
Q^{(t_1,\ldots,t_k)}_{a,b}=\frac{Q^{(a/b,t_1,\ldots,t_k)}_{0,0} - Q^{(t_1,\ldots,t_k)}_{0,0}}{q-1}.
\]
\end{lemma}
\begin{proof}
For the rest of this proof, let $t=(t_1,\ldots,t_k)$ and $t'=(a/b,t_1,\ldots,t_k)$, and let $u$ and $v'$ be shorthand for $(u_1,\ldots,u_k)$ and $(v_0,v_1,\ldots,v_k)$, respectively.    Then
\begin{align*}
Q^{t'}_{0,0} - Q^t_{0,0}
& = |\{v' \in K^{k+1}: v_0\not=0, t'\cdot v'=0, \ferm{v'} = 0\}| \\
& = |\{(v_0,u) \in K^\times \times K^k: t\cdot u=a, \ferm{u} = b\}| \\
& = (q-1) Q^t_{a,b},
\end{align*}
where the second equality uses the reparameterization with $u_j=-b v_j/v_0$ for $j \in \{1,\ldots,k\}$ and the fact that the invertibility of $s$ makes $(-1)^s=-1$.
\end{proof}
Now we compute certain values of $Q^t_{a,b}$ that will be useful later.
\begin{lemma}\label{Copernicium}
Let \(t_1, t_2 \in \Ku\) and let $\delta$ denote the Kronecker delta.
\begin{enumerate}[label=(\roman*)]
\item We have $Q^{(t_1)}_{a,b} = \delta_{a, t_1 b}$.
\item If at least one of $a$ or $b$ is zero, then 
\[
Q^{(t_1,t_2)}_{a,b} = \begin{cases}
1+(q-1) \delta_{t_1,t_2} & \text{if $a=b=0$,} \\
1-\delta_{t_1,t_2} & \text{otherwise.}
\end{cases}
\]
\end{enumerate}
\end{lemma}
\begin{proof}
The first claim is clear because $Q^{(t_1)}_{a,b}$ counts the number of $v_1 \in K$ such that $t_1 v_1 = a$ and $(v_1^s)^{1/s}=b$.  
Applying this result to the fact that $Q^{(t_1,t_2)}_{0,0} = (q-1) Q^{(t_2)}_{t_1,1}+Q^{(t_2)}_{0,0}$ by \cref{Nihonium} gives the expression in the first case of the second claim, and then the second case follows from using \cref{Flerovium} to deduce the value of $Q^{(t_1,t_2)}_{a,0}$ and $Q^{(t_1,t_2)}_{0,a}$.
\end{proof}
We explore certain special values of $Q^t_{a,b}$ that are critical for our proof of \cref{MainTheorem}.
\begin{lemma} \label{Roentgenium}
Let \(w \in K\). 
\begin{enumerate}[label=(\roman*)]
\item We have \(Q^{(1,-1)}_{1,w} = Q^{(1,-1)}_{1, -w}\).
\item If \(p\) is odd, then \(Q^{(1,-1)}_{1,1} - 1 = Q^{(1,-1)}_{1,-1} - 1 = Q^{(1,1)}_{1,-1}\).
\item When \(p\) is odd and \(w = 2^{1/s -1}\), then \(Q^{(1,1)}_{1,w}\) is odd; otherwise \(Q^{(1,1)}_{1,w}\) is even.
\end{enumerate}
\end{lemma}
\begin{proof}
Recall that \((-1)^s = -1\) since \(s\) is invertible.

The first result follows from the observation that \((x_1,x_2) \in K^2\) satisfies the system of equations corresponding to \(Q^{(1,-1)}_{1,w}\) if and only if \((-x_2,-x_1)\) satisfies the system of equations corresponding to \(Q^{(1,-1)}_{1, -w}\). 

Now $Q^{(1,1)}_{1,-1}$ counts how many $(x_1, x_2) \in K^2$ satisfy $x_1 + x_2 = 1$ and $x_1^s + x_2^s = (-1)^s$, and since these equations preclude $x_2=0$ in odd characteristic, we can reparameterize with $x_2=-1/y$ for $y \in \Ku$ and eliminate $x_1$ to see that $Q^{(1,1)}_{1,-1}$ is the same as the number of $y \in \Ku$ such that $(y+1)^s+y^s=1$, which is $Q^{(1,-1)}_{1,1}-1$ because $(0+1)^s+0^s=1$.

For the third result, note that the system of equations that corresponds to \(Q^{(1,1)}_{1,w}\) is symmetric in both unknowns, so \((x_1, x_2)\) satisfies this system if and only if \((x_2, x_1)\) does. 
This implies that \(Q^{(1,1)}_{1,w}\) is even except when there is some \(x \in K\) such that \(2 x = 1\) and \(2^{1/s} x = w\), which happens exactly when \(p\) is odd, \(x = 1/2\), and \(w = 2^{1/s - 1}\).
\end{proof}

\section{Group algebra} \label{ALGEBRAS}

In this section, we use a group algebra that gives us a convenient way to encapsulate all the Weil spectrum values in a single object; this builds upon the methods of Feng \cite{Feng} and developments in \cite{Katz-2015}.
After introducing the relevant group algebra here, we define the key group algebra elements of interest in \cref{HERBERT} and demonstrate their relation to the cardinalities of algebraic sets studied in \cref{SETS}. 
Then we present other related group algebra elements designed to have a particular symmetry in \cref{TOR}, and focus on a particularly important case of this symmetry in \cref{PHILIPPE}. 

Let \(L = \Q(\zeta,\xi)\), where \(\xi = \exp(2\pi i/(q-1))\), and consider the group \(L\)-algebra \(\gr\), whose elements are of the form \(S = \sum_{u \in \Ku} S_u [u]\), where \(S_u \in L\) for each \(u \in \Ku\).
We write the elements of \(\Ku\) in brackets to distinguish them from similar-appearing elements in \(L\).
We identify any subset \(U\) of \(\Ku\) with \(\sum_{u \in U} [u]\) in \(\gr\). 
For \(S = \sum_{u \in \Ku} S_u [u] \in \gr\), we define its conjugate to be \(\conj{S} = \sum_{u \in \Ku} \conj{S_u} [u^{-1}]\). 
We also let \(|S| = \sum_{u \in \Ku} S_u\); this is the cardinality of \(S\) if \(S\) is a group algebra element representing a subset of 
\(\Ku\).
Moreover, if \(t \in \Z\), we write \(S^{(t)}\) to denote \(\sum_{u \in \Ku} S_u [u^t]\).

Below, we record some easily proved observations. 
\begin{lemma}\label{Hydrogen}
For any \(S, T \in \gr\) and any \(t \in \Z\), we have 
\begin{enumerate}[label = (\roman*)]
\item \(|S^{(t)}| = |S|\);
\item \(|\conj{S}| = \conj{|S|}\);
\item \(|S + T| = |S| + |T|\);
\item\label{dace} \(|S T| = |S| |T|\);
\item \(S \Ku = |S| \Ku\);
\item if \(S\) is a subgroup of \(\Ku\), then \(\conj{S} = S^{(-1)} = S\) and \(S^2 = |S| S\); and
\item\label{gar} \((S \conj{S})_1 = \sum_{u \in \Ku} |S_u|^2\).
\end{enumerate}
\end{lemma}
Let \(\mchars\) denote the group of multiplicative characters from \(\Ku\) to \(\Lu\).
The identity element of \(\mchars\) is called the {\it principal character} and is written \(\chi_0\); it maps every element of \(\Ku\) to \(1\).
We define the application of a multiplicative character \(\chi\in\mchars\) to a group algebra element \(S = \sum_{u \in \Ku} S_u [u] \in \gr\) by linear extension:
\[
\chi(S) = \sum_{u \in \Ku} S_u \chi(u),
\]
and we call \(\chi(S)\) {\it the Fourier coefficient of \(S\) at \(\chi\)}.

The following facts, which we record without proof, are easy to verify. 
\begin{lemma}\label{Helium}
The following facts hold for any \(S, T \in \gr\) and any \(\chi \in \mchars\): 
\begin{enumerate}[label = (\roman*)]
\item \(\chi_0(S) = |S|\),
\item \(\chi(S^{(t)}) = \chi^t(S)\) for any \(t \in \Z\),
\item \(\chi(\conj{S}) = \conj{\chi(S)}\),
\item \(\chi(S+T) = \chi(S)+\chi(T)\), and 
\item \(\chi(S T) = \chi(S) \chi(T)\).
\end{enumerate}
\end{lemma}
The next lemma follows from Theorem 5.4 of \cite{Lidl-Niederreiter}.
\begin{lemma}\label{Lithium}
  We have 
  \[
  \sum_{u \in \Ku} \chi(u) = \chi(\Ku) = \begin{cases}
    q-1 & \text{ if } \chi = \chi_0 \\
    0 & \text{ otherwise}.
  \end{cases}
  \]
\end{lemma}
The next result says that a group algebra element is determined by its Fourier transform.
\begin{lemma}\label{Beryllium}
If \(S, T \in \gr\), then \(S=T\) if and only if \(\chi(S) = \chi(T)\) for all \(\chi \in \mchars\).
\end{lemma}
\begin{proof}
This follows from the fact that the Fourier transform (the map from \(\gr\) to \(L^{\mchars}\) that takes \(S\) to the function \(\widehat{S}\colon \mchars \to L\) with \(\widehat{S}(\chi)=\chi(S)\)) is an isomorphism of \(L\)-algebras with the inverse map 
\begin{align*}
L^{\mchars}	& \to \gr \\
R  	& \mapsto \widecheck{R} = \sum_{u \in \Ku} \widecheck{R}_u [u],
\end{align*}
where 
\[
\widecheck{R}_u = \frac{1}{|\Ku|} \sum_{\chi \in \mchars} R(\chi) \conj{\chi(u)}. \qedhere
\]
\end{proof}

\subsection{Weil sums in the group algebra}\label{HERBERT}

Recall that \(\psi : K \to \Q(\zeta)\) is the canonical additive character of \(K\). 
We define
\[
\Psi = \sum_{u \in \Ku} \psi(u) [u]
\]
and
\begin{equation}\label{celery}
\weil = \sum_{u \in \Ku} \weil_u [u],
\end{equation}
and when the field \(K\) and the exponent \(s\) are clear from context, we simply write \(W = \sum_{u \in \Ku} W_u [u]\).
We now relate $W$ to $\Psi$.
\begin{lemma} \label{Boron}
We have \(W = \Psi \conj{\Psi^{(1/s)}} + \Ku\).
\end{lemma}
\begin{proof}
Applying the reparameterization \(z = -x^s, y = -u x\) to \(\Psi \conj{\Psi^{(1/s)}} = \sum_{y, z \in \Ku} \psi(y) \conj{\psi(z)} [y z^{-1/s}] \) gives 
\begin{align*}
\Psi \conj{\Psi^{(1/s)}} = \sum_{u, x \in \Ku} \psi(x^s) \psi(-u x) [u] = \sum_{u \in \Ku} (W_u - 1) [u],
\end{align*}
from which the result follows.
\end{proof}

Let \(\chi \in \mchars\).
Then the \textit{Gauss sum} \(G(\chi)\) is given by \(\sum_{u \in \Ku} \psi(u) \chi(u)\). 
Note that \(G(\chi) = \chi(\Psi)\).
We list some useful facts about Gauss sums, which will be useful later when we calculate the Fourier transform of group algebra elements that generalize $W$.
\begin{lemma}\label{Carbon}
Let \(\chi_0\) be the principal character and \(\chi \in \mchars\). 
Then 
\begin{enumerate}[label=(\roman*)]
\item $|\Psi| = \chi_0(\Psi) = G(\chi_0)=-1$,
\item $|\chi(\Psi)| = |G(\chi)|=\sqrt{q}$ for $\chi\not=\chi_0$, and
\item $\conj{G(\chi)}=\chi(-1) G(\conj{\chi})$.
\end{enumerate} 
\end{lemma}
\begin{proof}
For a proof of the first and second parts, see \cite[Theorem 5.11]{Lidl-Niederreiter}; for a proof of the third part, see \cite[Theorem 5.12(iii)]{Lidl-Niederreiter}.
\end{proof}
We record two more useful calculations concerning $\Psi$ and $W$.
\begin{lemma} \label{Nitrogen}
If \(t \in \Z\) and \(\gcd(t, q-1) = 1\), then \(\Psi^{(t)} \conj{\Psi^{(t)}} = q[1] - \Ku\).
\end{lemma}
\begin{proof}
See \cite[Corollary 2.3]{Katz-2015}.
\end{proof}

\begin{lemma} \label{Oxygen}
We have \(|W| = q\) and \(W \conj{W} = q^2[1]\).
\end{lemma}
\begin{proof}
The first result follows from Lemmas \ref{Boron}, \ref{Hydrogen}, and \ref{Carbon}, which give us 
\(|W| = |\Psi| |\conj{\Psi^{(1/s)}}| + |\Ku| = (-1)^2 + q-1\). 
The second is proved as follows: 
\begin{align*}
W \conj{W}
& = (\Psi \conj{\Psi^{(1/s)}} + \Ku)(\conj{\Psi \conj{\Psi^{(1/s)}} + \Ku}) \\
& = \Psi \conj{\Psi} \Psi^{(1/s)} \conj{\Psi^{(1/s)}} + \Psi \conj{\Psi^{(1/s)}} \Ku + \conj{\Psi} \Psi^{(1/s)} \Ku + \Ku \Ku \\
& = (|K|[1] - \Ku)^2 + |\Psi \conj{\Psi^{(1/s)}}| \Ku + |\conj{\Psi} \Psi^{(1/s)}| \Ku + |\Ku| \Ku \\
& = |K|^2[1] - 2 |K| \Ku + |\Ku| \Ku + (|\Ku| + 2) \Ku \\
& = |K|^2[1],
\end{align*}
where the third equality uses Lemmas \ref{Nitrogen} and \ref{Hydrogen}, and the fourth equality uses Lemmas \ref{Hydrogen} and \ref{Carbon}.
\end{proof}

The proof of our main result requires us to use generalizations of $W$ whose coefficients are products of Weil sum values rather than individual ones.  
We introduce a convenient notation for these. 

\begin{notation}\label{Winifred}
Let \(k \in \Z_+\) and let \(t=(t_1, t_2, \ldots, t_k) \in \Kuk\). We write
\[
W^{[t]} = \sum_{u \in \Ku} W_{t_1 u} \cdots W_{t_k u} [u].
\]
Often, we just write \(W^{[t_1, \ldots, t_k]}\) instead of \(W^{[(t_1, \ldots, t_k)]}\) and \(W^{[t]}_u\) for \((W^{[t]})_u\).
Also, note that \(W^{[1]} = W\).
\end{notation}

\cref{Neon} and \cref{Aluminum} below make a connection between the group algebra elements just defined in \cref{Winifred} and the cardinalities of algebraic sets defined in \cref{Tennessine}.
The connecting object is defined in \cref{Victor}.
\begin{lemma}\label{Neon}
Let \(k \in \Z_+\) and let \(t = (t_1, \ldots, t_k) \in \Kuk\). Then
\[
|W^{[t]}| = \sum_{u \in \Ku} W_{t_1 u} \cdots W_{t_k u} = \frac{q^2 Q^{t}_{0,0} - q^k}{q-1}.
\]
\end{lemma}
\begin{proof}
This is Lemma 7.7.2 of \cite{Katz-2019}.
\end{proof}

Recall the notations $\cdot$ and $\ferm{\cdot}$ from \cref{Oganesson}, which we use for the rest of this section.
\begin{notation}\label{Victor}
Let \(k \in \Z_+\) and let \(t = (t_1, t_2, \ldots, t_k) \in \Kuk\). 
Then, adopting the convention that \([0]\) is the \(0\) of the group algebra \(\gr\), we write 
\[
V^{[t]} = \sums{v \in K^k \\ t \cdot v = 1} [\ferm{v}] - Q^{t}_{1,0} \Ku,
\]
so that
\[
V^{[t]} = \sum_{u \in \Ku} (Q^t_{1,u}- Q^t_{1,0}) [u].
\]
We often write \(V^{[t_1, \ldots, t_k]}\) instead of \(V^{[(t_1, \ldots, t_k)]}\) and use the notation \(V^{[t]}_u\) to mean \((V^{[t]})_u\).
\end{notation}
The following calculation is needed for our proof of \cref{Aluminum}, which connects $W^{[t]}$ to $V^{[t]}$.
\begin{lemma}\label{Magnesium}
Let \(k \in \Z_+\) and \(t \in \Kuk\). 
Then 
\[
|V^{[t]}| = q^{k-1} - q \cdot Q^t_{1,0} = \frac{q Q^t_{0,0} - q^{k-1}}{q-1}.
\]
\end{lemma}
\begin{proof}
The first equality comes from using \cref{Victor} to write  \(|V^{[t]}| = \sum_{u \in \Ku} Q^t_{1,u} - (q-1) \cdot Q^t_{1,0} = \sum_{u \in K} Q^t_{1,u} - q \cdot Q^t_{1,0}\) and then applying \cref{Livermorium}. 
The second equality then follows from \cref{Flerovium}.
\end{proof}
Now we show the relation between $V^{[t]}$ and $W^{[t]}$.
\begin{proposition}\label{Aluminum}
For \(k \in \Z_+\) and \(t \in \Kuk\), we have 
\[
W^{[t]} = W V^{[t]}.
\]
\end{proposition}
\begin{proof}
Since both sides of this equation are elements of \(\gr\), it suffices to show that \(\chi(W^{[t]}) = \chi(W V^{[t]})\) for all \(\chi \in \mchars\) by \cref{Beryllium}.

For the principal character \(\chi_0\) we have
\[
\chi_0(W V^{[t]}) = |W| |V^{[t]}| = \frac{q^2 Q^t_{0,0} - q^k}{q-1} = |W^{[t]}| = \chi_0(W^{[t]}),
\]
where the first and last equalities follow from \cref{Helium}, the second comes from Lemmas \ref{Oxygen} and \ref{Magnesium}, and the third comes from \cref{Neon}.

Now let \(\chi\) be any non-principal character. 
On one hand, we have 
\begin{align*}
\chi\left(W^{[t]}\right) & = \sum_{u \in \Ku} \left(\sum_{x_1 \in K} \psi(x_1^s - t_1 u x_1)\right) \cdots \left(\sum_{x_k \in K} \psi(x_k^s - t_k u x_k)\right) \chi(u) \\
& = \sums{x \in K^k \\ t \cdot x \neq 0} \psi(\ferms{x}) \sum_{u \in \Ku} \psi(-(t \cdot x) u) \chi(u) \\
& = G(\chi) \sum_{w \in \Ku} \sums{x \in K^k \\ t \cdot x = w} \psi(\ferms{x}) \chi^{-1}(-w) \\
& = G(\chi) \sum_{z \in \Ku} \sums{v \in K^k \\ t \cdot v = 1} \psi( \ferms{v} z) \chi^{-1}(-z^{1/s}) \\
& = G(\chi) \sums{v \in K^k \\ t \cdot v = 1 \\ \ferm{v} \neq 0} \sum_{z \in \Ku} \psi(\ferms{v} z) \chi^{-1/s}(-1) \chi^{-1/s}(z) \\
& = G(\chi) G(\chi^{-1/s}) \chi^{-1/s}(-1) \sums{v \in K^k \\ t \cdot v = 1 \\ \ferm{v} \neq 0} \chi^{1/s}(\ferms{v}) \\
& = G(\chi) G(\chi^{-1/s}) \chi^{-1/s}(-1) \chi\left(\sums{v \in K^k \\ t \cdot v = 1} [\ferm{v}]\right),
\end{align*}
where we use \cref{Lithium} to impose $t\cdot x\not=0$ following the second equals sign, and we use the Gauss sum in the third and second-to-last equalities and the reparameterization \(w=z^{1/s}, x=z^{1/s} v\) in the fourth equality.

On the other hand, Lemmas \ref{Boron}, \ref{Helium}, \ref{Lithium}, and \cref{Victor} give us  
\[
\chi(W V^{[t]}) = \chi(\Psi) \chi(\conj{\Psi^{(1/s})}) \cdot \chi\left(\sums{v \in K^k \\ t \cdot v = 1} [\ferm{v}]\right),
\]
where
\[
\chi(\Psi) \chi(\conj{\Psi^{(1/s)}})  = G(\chi) \conj{\chi^{1/s}(\Psi)} = G(\chi) \conj{G(\chi^{1/s})} = G(\chi) G(\chi^{-1/s}) \chi^{-1/s}(-1)
\]
by Lemmas \ref{Carbon} and \ref{Helium}, so the result is proved. 
\end{proof}
For future convenience, we explicitly calculate some values of \(|W^{[t]}|\) and \(V^{[t]}\).
\begin{lemma} \label{Silicon}
For any \(t_1, t_2, t_3 \in \Ku\), we have 
\begin{enumerate}[label=(\roman*)]
\item \(|W| = q\),
\item \(|W^{[t_1,t_2]}| = \begin{cases} 
q^2 & \text{if } t_1 = t_2 \\ 
0 & \text{if } t_1 \neq t_2, 
\end{cases} \)
\item\label{Lemuel} \(|W^{[t_1,t_2,t_3]}| = q^2 V^{[t_1,t_2]}_{1/t_3}\), and
\item \(|W^{[1,1,1,1]}| = q^2 \sum_{u \in \Ku} \left(V^{[1,1]}_u\right)^2\).
\end{enumerate}
\end{lemma}
\begin{proof}
The first result is from \cref{Oxygen}.

Next, we use \cref{Real} and \cref{Oxygen} to obtain  
\[
|W^{[t_1,t_2]}| = \sums{x, y \in \Ku \\ x y = t_1/t_2} W_x \conj{W_{1/y}}
= (W \conj{W})_{t_1/t_2}
=  \begin{cases}
  q^2 & \text{ if } t_1/t_2=1 \\
  0 	 & \text{otherwise}.
\end{cases}
\]

For \ref{Lemuel}, we use \cref{Real} to show that \(|W^{[t_1,t_2,t_3]}| = \left(W^{[t_1,t_2]} \conj{W}\right)_{1/t_3}\).
Then, Lemmas \ref{Aluminum} and \ref{Oxygen} and \cref{Victor} give us that 
\[
\left(W^{[t_1,t_2]} \conj{W}\right)_{1/t_3} = \left(W V^{[t_1,t_2]} \conj{W}\right)_{1/t_3} = q^2 V^{[t_1,t_2]}_{1/t_3}.
\]

Lastly, we observe that \(\sum_{u \in \Ku} W_u^4 = (W^{[1,1]} \conj{W^{[1,1]}})_1\), so the fourth result follows from Lemmas \ref{Aluminum} and \ref{Oxygen}, which tell us that 
\[
W^{[1,1]} \conj{W^{[1,1]}}= (W V^{[1,1]}) (\conj{W V^{[1,1]}}) = q^2 V^{[1,1]} \conj{V^{[1,1]}}. \qedhere
\] 
\end{proof}

\begin{lemma}\label{Phosphorus}
If $t=(t_1,t_2) \in (\Ku)^2$, then we have 
\[
V^{[t]} = \begin{cases}
  \sums{v \in K^2 \\ v \cdot t = 1} [\ferm{v}] & \text{if } t_1=t_2 \\
  \sums{v \in K^2 \\ v \cdot t = 1} [\ferm{v}] - \Ku & \text{otherwise,}
\end{cases}
\]
that is,
\[
V^{[t]} = \begin{cases}
\sum_{u \in \Ku} Q^t_{1,u} [u]  & \text{if } t_1=t_2 \\
\sum_{u \in \Ku} (Q^t_{1,u}-1) [u] & \text{otherwise,}
\end{cases}
\]
so that $V^{[t]}_u \geq -1$ for every $u \in \Ku$, and if $t_1=t_2$ then $V^{[t]}_u \geq 0$ for every $u \in \Ku$.
Furthermore, 
\[
|V^{[t]}| = \begin{cases}
  q & \text{if } t_1=t_2 \\
  0 & \text{otherwise}.
\end{cases}
\]
\end{lemma}
\begin{proof}
These facts follow from \cref{Victor} and \cref{Magnesium}, as well as the formula for \(Q^t_{1,0}\) found in \cref{Copernicium} and the fact that $Q^t_{1,u}$ values are always nonnegative, since they count solutions to systems of equations.
\end{proof}

\subsection{Symmetrized Weil sums} \label{TOR}

In later sections we study Weil spectra where there is a symmetry among the Weil sums \(\weil_u\). 
Here we present some general results. 

Fix some \(k \in \Z_+\) and suppose that \(p \equiv 1 \pmod{k}\).
Then we let
\[
T = \sum_{i=0}^{k-1} [\lambda^i],
\]
where \(\lambda\) is a primitive \(k\)th root of unity in \(\Fpu\). 
We also let 
\[
\Omega = \sum_{u \in \Ku} \left(\sum_{i=0}^{k-1} W_{\lambda^i u}\right) [u].
\]
We call $\Omega_u=\sum_{i=0}^{k-1} W_{\lambda^i u}$ {\it the $k$-laterally symmetrized Weil sum at $u$}, and we use the word {\it bilateral} to mean $2$-lateral.
Note that $\Omega$ has real coefficients by \cref{Real}.
First, we relate $\Omega$ to $W$ and $T$.
\begin{lemma} \label{Sulfur}
We have \(\Omega = W T\).
\end{lemma}
\begin{proof}
Reordering the sums in the definition of \(\Omega\) gives us
\[
\Omega = \sum_{i=0}^{k-1} \sum_{u \in \Ku} W_{\lambda^i u} [u] 
	   = \sum_{i=0}^{k-1} W \cdot [\lambda^{-i}] 
	   = W \cdot \sum_{i=0}^{k-1} [\lambda^i]. \qedhere
\]
\end{proof}
We compute power moments for $\Omega$.
\begin{lemma}\label{Chlorine}
We have 
\begin{enumerate}[label=(\roman*)]
\item\label{alpaca} \(\sum_{u \in \Ku} \Omega_u^0 = q-1\);
\item\label{buffalo} \(\sum_{u \in \Ku} \Omega_u = k q\);
\item\label{caribou} \(\sum_{u \in \Ku} \Omega_u^2 = k q^2\); and
\item if \(k=2\), then \(\sum_{u \in \Ku} \Omega_u^3 = 2 q^2 (V^{[1,1]}_1 + 3 V^{[1,1]}_{-1})\).
\end{enumerate}
\end{lemma}
\begin{proof}
The first equation comes from the fact that \(|\Ku| = q-1\).
The second and third results follow from Lemmas \ref{Sulfur} and \ref{Oxygen} as well as the fact that the coefficients of \(\Omega\) are real, so that 
\(
\sum_{u \in \Ku} \Omega_u = |\Omega| = |W| |T| = q k
\)
and 
\(
\sum_{u \in \Ku} \Omega_u^2 = (\Omega \conj{\Omega})_1 = (W T \conj{W} \conj{T})_1 = q^2 (T^2)_1 = q^2 k. 
\)
Lastly, if \(k=2\), then 
\begin{align*}
\sum_{u \in \Ku} \Omega_u^3 = \sum_{u \in \Ku} (W_u + W_{-u})^3 = 2 \sum_{u \in \Ku} W_u^3 + 6 \sum_{u \in \Ku} W_u^2 W_{-u}, 
\end{align*}
so the desired result follows from \cref{Silicon}\ref{Lemuel}.
\end{proof}

When \(p\) is odd and \(k=2\), we have further results, which we explore in the next section.

\subsection{Bilateral symmetry in the group algebra}\label{PHILIPPE}

In Propositions \ref{Enoch} and \ref{Barnabas} below we study bilaterally symmetrized Weil sums.
Here, we present some general results that hold in this situation.
To this end, suppose that \(p\) is odd and let
\begin{align*}
S = [1] - [-1],  \quad\,\,\, & \Phi = \sum_{u \in \Ku} (W_u - W_{-u}) [u], \,\,\, && \Upsilon = \sum_{u \in \Ku} (W_u - W_{-u})^2 [u], \\
T = [1] + [-1], \quad\,\,\, & \Omega = \sum_{u \in \Ku} (W_u +W_{-u}) [u].
\end{align*}
Note that this use of \(T\) and \(\Omega\) is consistent with the notation introduced in \cref{TOR} when \(k=2\).
Also, note that \(\Phi\), \(\Omega\), and \(\Upsilon\) have real coefficients by \cref{Real}.
For convenience of notation, we set 
\[
V = V^{[1,1]} \qquad \text{ and } \qquad U = V^{[1,-1]}.
\]
We relate the various group algebra elements that we have just defined.
\begin{lemma}\label{Argon}
We have \(\Phi = W S\) and \(\Upsilon = W(T V - 2 U)\).
\end{lemma}
\begin{proof}
These results come from the above notation and \cref{Aluminum}.
\end{proof}

Before we prove further results, we shall restate in the notation of this section a few key facts that we have proved earlier.
\begin{lemma}\label{Potassium}
We have 
\begin{enumerate}[label=(\roman*)]
\item\label{kangaroo} \(U_u = Q^{(1,-1)}_{1,u} - 1\) and \(V_u = Q^{(1,1)}_{1,u}\) so in particular, \(U_u \geq -1\) and \(V_u \geq 0\) for all \(u \in \Ku\);
\item\label{koala} \(U_u = U_{-u}\) for any \(u \in \Ku\);
\item\label{wallaby} \(U_1 = U_{-1} = V_{-1}\); and
\item\label{wombat} \(\sum_{u \in \Ku} U_u = 0\) and \(\sum_{u \in \Ku} V_u = q\).
\end{enumerate}
\end{lemma}
\begin{proof}
Part \ref{kangaroo} follows from the definitions of \(U\) and \(V\) and \cref{Phosphorus}. 
Then, using the result in part \ref{kangaroo} and the assumption that \(p\) is odd, parts \ref{koala} and \ref{wallaby} follow from the first two parts of \cref{Roentgenium}.
Lastly, the part \ref{wombat} comes from \cref{Phosphorus}.
\end{proof}
We compute some power moments for $\Phi$ and a related sum that involves both $\Phi$ and $\Omega$.
\begin{lemma}\label{Calcium}
We have 
\begin{enumerate}[label = (\roman*)]
\item \(\sum_{u \in \Ku} \Phi_u = 0\),
\item\label{gnu} \(\sum_{u \in \Ku} \Phi_u^2 = 2 q^2\), 
\item \(\sum_{u \in \Ku} \Phi_u^2 \cdot \Omega_u = 2 q^2 (V_1 - V_{-1})\), and 
\item\label{ibex} \(\sum_{u \in \Ku} \Phi_u^4 = q^2 \sum_{u \in \Ku} (V_u + V_{-u} - 2 U_u)^2\).
\end{enumerate}
\end{lemma}
\begin{proof}
Recall that \(\Phi\) and \(\Upsilon\) have real coefficients.

To prove the first part, we use Lemmas \ref{Argon} and \ref{Hydrogen}\ref{dace} to get \(|\Phi| = |W| |S| = 0\).
The second part follows from Lemmas \ref{Argon}, \ref{Oxygen}, and \ref{Hydrogen}\ref{gar}, which give us \(
\sum_{u \in \Ku} \Phi_u^2 = (\Phi \conj{\Phi})_1 = (W S \conj{W} \conj{S})_1 = q^2 (S \conj{S})_1 = 2 q^2. 
\)

We can prove the third part by observing that \(\sum_{u \in \Ku} \Phi_u^2 \cdot \Omega_u = (\Upsilon \cdot \conj{\Omega})_1\) and then using Lemmas \ref{Argon}, \ref{Sulfur}, \ref{Oxygen}, and \ref{Potassium}\ref{wallaby} to get that 
\[
(\Upsilon \cdot \conj{\Omega})_1 = \left(W \conj{W} (2 T V - 2 U T)\right)_1 = 2 q^2 (V_1 + V_{-1} - U_1 - U_{-1}) = 2 q^2 (V_1 - V_{-1}). 
\]

The fourth and final part is a consequence of Lemmas \ref{Hydrogen}\ref{gar} (using the fact that all coefficients in our group algebra elements here are real), \ref{Argon} and \ref{Oxygen}, since we have 
\begin{align*}
\sum_{u \in \Ku} \Phi_u^4 & = (\Upsilon \conj{\Upsilon})_1 \\
						  & = \left(W \conj{W} (T V - 2 U) (\conj{T V - 2 U})\right)_1 \\
						  & = q^2 \sum_{u \in \Ku} \left((T V - 2 U)_u\right)^2 \\
						  & = q^2 \sum_{u \in \Ku} (V_u + V_{-u} - 2 U_u)^2. \qedhere
\end{align*}
\end{proof}

\section{Cyclotomic actions on value sets of size four}\label{ACTIONS}

In this section, we examine the action on the value set $\ws$ (see \eqref{William}) of $\tau$, the restriction of the generator $\sigma$ of $\Gal(\Qz/\Q)$ to $\ws$.
We shall prove our main theorem (\cref{MainTheorem}), which is: 
\begin{theorem}\label{Juvenal}
Let $K$ be a finite field and $s$ be an invertible exponent over $K$.
If the Weil spectrum for $K$ and $s$ is $4$-valued, then it is rational unless \(K = \F_5\) and \(s \equiv 3 \pmod{4}\) (in which case \(\ws = \{(5 \pm \sqrt{5})/2, \pm \sqrt{5}\}\)).
\end{theorem}

Suppose \(\ws = \{A,B,C,D\}\), where \(A\), \(B\), \(C\), and \(D\) are distinct. 
Recall that \(\sigma\) is a generator of \(\Gal(\Qz/\Q)\) and that \(\tau\) is the restriction of \(\sigma\) to \(\ws\). 
We saw in \eqref{Rachel} that \(\tau\) always permutes the elements of \(\ws\), so here \(\tau\) must act trivially, as a transposition (while keeping two values fixed), as a composition of two disjoint transpositions, as a \(3\)-cycle (while keeping one value fixed), or as a \(4\)-cycle on the set \(\{A,B,C,D\}\).
We shall address each of the non-trivial actions in the next four propositions, and then finally prove the theorem.
Throughout this section, we shall use the notation (from \eqref{Nora} in the Introduction) where $N^{K,s}_A$ (or simply $N_A$) denotes the frequency of a value $A$ in the Weil spectrum for the field $K$ and the exponent $s$.

\begin{proposition}[{\bf No action as a \(4\)-cycle}] \label{Josiah}
If \(|\ws| = 4\), then \(\tau\) does not permute \(\ws\) as a \(4\)-cycle.
\end{proposition}
\begin{proof}
This follows from \cref{Naphtali}, since \(|\ws| = 1\) when \(K = \F_2\).
\end{proof}

\begin{proposition}[{\bf No action as a \(3\)-cycle}] \label{Paul}
If \(|\ws| = 4\), then \(\tau\) does not permute \(\ws\) as a \(3\)-cycle (while fixing one value).
\end{proposition}
Let \(|\ws| = 4\).
Assume that \(\tau\) permutes $\ws$ as a $3$-cycle to show a contradiction, so we write \(\ws=\{A,B,C,D\}\) and \(\tau=(A)(B C D)\), i.e., $\tau(A)=A$, $\tau(B)=C$, $\tau(C)=D$, and $\tau(D)=B$.
For clarity, we break the proof into steps. 
\begin{step}\label{Aleph-1}
The exponent \(s\) is nondegenerate.
\end{step}
\begin{proof}\let\qed\relax
This is from \cref{Tor}.
\end{proof}
\begin{step}\label{Aleph-2}
We have \(p \equiv 1 \pmod{6}\), so \(p \geq 7\), and there is a primitive third root of unity \(\lambda\in\Fpu\) such that \(\tau(W_u) = W_{\lambda u}\) for all \(u \in \Ku\).
\end{step}
\begin{proof} \let\qed\relax
This is from \cref{Simeon} and \eqref{Rachel}, since  \(\tau\) has order \(3\).
\end{proof}
\begin{step}\label{Aleph-3}
We have \(3 \mid N_A\) and \(N_B = N_C = N_D\).
\end{step}
\begin{proof}\let\qed\relax
This is from \cref{Judah}, since  \(\tau\) has order \(3\),  permutes \(A\) in a \(1\)-cycle, and permutes \(B, C, D\) in a \(3\)-cycle.
\end{proof}
\begin{step} \label{Bet-1}
Let \(X = 3 A\) and \(Y = B + C + D\). 
Then \(X\) and \(Y\) are rational integers with \(3\mid X\) and
\begin{equation}\label{East}
3 q^2 - 3 q(X+Y) + (q-1)X Y = 0.
\end{equation}
\end{step}
\begin{proof} \let\qed\relax
We know that \(A\) and \(Y\) are in \(\Z\) because $\sigma$ (of which $\tau$ is a restriction) fixes both of these algebraic integers.
Thus, \(X\) is a rational integer with \(3\mid X\).
By \cref{Aleph-2}, we can let \(\Omega_u = W_u + W_{\lambda u} + W_{\lambda^2 u} = W_u + \tau(W_u) + \tau^2(W_u)\) for all \(u \in \Ku\), as in \cref{TOR} (with \(k=3\)).
Notice that \(\Omega_u\) only assumes two values as \(u\) runs through \(\Ku\), namely \(X = 3 A\) (\(N_A\) times) and \(Y = B + C + D\) (\(3 N_B\) times by \cref{Aleph-3}).
This means that   
\(
\sum_{u \in \Ku} (\Omega_u - X)(\Omega_u - Y) = 0,
\)
so we obtain \eqref{East} from the first three results in \cref{Chlorine}. 
\end{proof}
\begin{step}\label{Bet-2}
We have \(\max\{v_p(X), v_p(Y)\} \geq v_p(q)\).
\end{step}
\begin{proof}\let\qed\relax
If \(\max\{v_p(X), v_p(Y)\} < v_p(q)\), then \(v_p((q-1) X Y) < v_p(3 q^2 - 3 q (X + Y))\), contradicting \eqref{East} in \cref{Bet-1}.
\end{proof}
\begin{step}\label{Gimel}
We have \(0 \notin \{X,Y\}\).
\end{step}
\begin{proof}\let\qed\relax
We assume \(0 \in \{X,Y\}\) to show contradiction.
Then \(\{X,Y\}=\{0,q\}\) by \eqref{East}.
Now \(q\) is a power of the prime \(p\) with \(p \geq 7\) (by \cref{Aleph-2}), but \(X\) is a rational integer with \(3\mid X\) (by \cref{Bet-1}), so we cannot have \(X=q\).
Thus, \(X = 3 A = 0\) and \(Y = B + C + D = q\).
Since \(s\) is nondegenerate by \cref{Aleph-1}, we have \(|B|, |C|, |D| < q\) by \cref{Asher}, so that $B+C+D=q$ makes at least two of $B$, $C$, $D$ positive, while \cite[Corollary 2.3]{Aubry-Katz-Langevin} makes at least one negative, and so  $B C D < 0$.
Now \cref{Silicon}\ref{Lemuel} gives us $|W^{[1,\lambda,\lambda^2]}|=q^2 V^{[1,\lambda]}_\lambda$, that is,
$\sum_{u \in \Ku} W_u W_{\lambda u} W_{\lambda^2 u}=q^2 V^{[1,\lambda]}_\lambda$.
Recalling the relation involving $\tau$ and $\lambda$ from \cref{Aleph-2}, this means that $\sum_{u \in \Ku} W_u \, \tau(W_u) \, \tau^2(W_u)=q^2 V^{[1,\lambda]}_\lambda$.
Then in view of the fact that \(\ws=\{A,B,C,D\}\) with $\tau(B)=C$, $\tau(C)=D$, and $\tau(D)=B$, and since we have shown that \(A=0\) here and $N_B=N_C=N_D$ in \cref{Aleph-3}, we have $3 N_B B C D = q^2  V^{[1,\lambda]}_\lambda$.  Since $B C D < 0$, \cref{Phosphorus} forces $V^{[1,\lambda]}_\lambda = -1$, and hence $3 N_B B C D = -q^2$.  But \(B C D \in \Z\) since it is an algebraic integer fixed by \(\tau\), which is a restriction of \(\sigma\), the generator of $\Gal(\Qz/\Q)$.
This means that \(3 \mid q^2\), contradicting $p \geq 7$ from \cref{Aleph-2}.
\end{proof}
\begin{step}\label{Dalet-1}
We have \(v_p(X) < v_p(q)\) and \(v_p(Y) \geq v_p(q)\).
\end{step}
\begin{proof}\let\qed\relax
Recall from \cref{Bet-1} that \(X=3 A\).
Therefore, by \cref{Aleph-2} we have \(v_p(X) = v_p(3 A) = v_p(A)\), and then by \cref{Gimel} and \cref{Zebulun}, we know that \(v_p(A) < v_p(q)\).
Thus \(v_p(X) < v_p(q)\) and so by \cref{Bet-2} we know that \(v_p(Y) \geq v_p(q)\).
\end{proof}
\begin{step}\label{Dalet-2}
We have \(Y = r q\) for some \(r \in \{\pm 1, \pm 2\}\).
\end{step}
\begin{proof}\let\qed\relax
Recall from \cref{Bet-1} that \(Y=B+C+D\).
By Steps \ref{Gimel} and \ref{Aleph-1} combined with \cref{Asher}, we know that \(0 < |Y| = |B+C+D| < 3 q\).
Thus, \(0 < |Y| < p q\) by \cref{Aleph-2}.
Now \cref{Bet-1} shows that \(Y \in \Z\), so by \cref{Dalet-1} we have \(v_p(Y) = v_p(q)\).
Then $Y=r q$ for some $r \in \Z$ and recall that $0 < |Y| < 3 q$.
\end{proof}
\begin{step}
We conclude that \(\tau\) does not permute \(\ws\) as a \(3\)-cycle.
\end{step}
\begin{proof}
We rule out each of the four possible values of $r$ in \cref{Dalet-2} using the following formula for \(X \in \Z\), which comes from \eqref{East} and \(Y=r q\):
\begin{equation}\label{South}
X = \frac{3 q(r-1)}{(q-1)r - 3}
\end{equation}
(note that the denominator is not zero because  \(q-1 \geq p-1 \geq 6\) by \cref{Aleph-2}).
 
If \(r=1\), then \(X=0\), which contradicts \cref{Gimel}.
If $r=-2$ (resp., $-1$, $2$), then \eqref{South} becomes $4+(q-4)/(2 q+1)$ (resp., $6-12/(q+2)$, $1+(q+5)/(2 q-5)$).
None of these expressions can be a rational integer, since $q$ is a power of some prime $p \geq 7$ by \cref{Aleph-2}, so \cref{Bet-1} is contradicted.
\end{proof}

\begin{proposition}[{\bf Action as a composition of two disjoint \(2\)-cycles}] \label{Enoch}
The following are equivalent:
\begin{enumerate}[label = (\roman*)]
\item \(|\ws| = 4\) and \(\tau\) permutes \(\ws\) as a composition of two disjoint transpositions;
\item \(q = 5\) and \(s \equiv 3 \pmod{4}\).
\end{enumerate}
When these hold, \(\ws = \{(5 \pm \sqrt{5})/2, \pm \sqrt{5}\}\).
\end{proposition}
Suppose that $q=5$ and $s \equiv 3 \pmod{4}$.
In fact, we may assume \(s = 3\) since \(\cW_{K, s'} = \cW_{K, s''}\) if \(s' \equiv s'' \pmod{q-1}\) (see the definition of equivalent exponents in \cref{INTRODUCTION}). 
Let \(\zeta = e^{2 \pi i/5}\).
The polynomial \(x^3 - x\) represents \(0\) thrice and each of \(\pm 1\) only once over \(K = \F_5\), and so \(\weil_1 = 3 + (\zeta + \zeta^{-1}) = (5 + \sqrt{5})/2\) by \cref{Joseph}. 
Similarly, \(x^3 - 2 x\) represents \(0\) once and each of \(\pm 1\) twice over \(K\); \(x^3 - 3 x\) represents \(0\) once and each of \(\pm 2\) twice; and \(x^3 - 4 x\) represents \(0\) three times and each of \(\pm 2\) once over \(K\), so we can calculate that \(\weil_2 = \sqrt{5}\), \(\weil_3 = - \sqrt{5}\), and \(\weil_4 = (5 - \sqrt{5})/2\).
Then since $\sigma(\sqrt{5})=-\sqrt{5}$ (because $\sigma$ restricts to the generator of $\Gal(\Q(\sqrt{5})/\Q)$), it is clear that $\tau$ acts on $\ws$ as a product of two disjoint transpositions.

Now suppose that \(|\ws| = 4\) and suppose that \(\tau\) acts on \(\ws\) as a composition of two disjoint transpositions.
For clarity, the remainder of the proof is broken up into steps.
\setcounter{step}{0}
\renewcommand\theHstep{Enoch.\arabic{step}}
\begin{step}\label{Alpha-1}
We have \(p \equiv 1 \pmod{4}\), so \(p\geq 5\), and \(\tau(W_u) = W_{-u}\) for all \(u \in \Ku\).
\end{step}
\begin{proof} \let\qed\relax
This is from \cref{Simeon} and \eqref{Rachel}, since  \(\tau\) has order \(2\).
\end{proof}
\begin{step}\label{Alpha-2}
We write $\ws=\{A,B,C,D\}$ with
\[
A = \frac{E + F \sqrt{p}}{2}, B = \frac{E - F \sqrt{p}}{2}, C = \frac{G + H \sqrt{p}}{2}, \text{ and } D = \frac{G - H \sqrt{p}}{2},
\]
where \(E, F, G, H \in \Z\) with \(E \equiv F \pmod{2}\), \(G \equiv H \pmod{2}\), $v_p(E) \leq v_p(G)$, and $\tau=(A B)(C D)$, i.e., \(\tau(A) = B\), \(\tau(B) = A\), \(\tau(C) = D\), and \(\tau(D) = C\).
\end{step}
\begin{proof} \let\qed\relax
Since \(\tau\) has order \(2\) and since \cref{Alpha-1} tells us that $p\equiv 1 \pmod{4}$, \cref{Simeon} shows that the elements of $\ws$ are algebraic integers in \(\Q(\sqrt{p})\), the unique degree \(2\) extension of \(\Q\) that lies in \(\Qz\).
Thus, each element of $\ws$ has the form described in \cref{Benjamin}, and the four elements consist of two pairs of Galois conjugates because of the action of $\tau$.
This establishes the existence of the integers $E$, $F$, $G$, and $H$ which are used to describe our four elements of $\ws$ above (making sure to arrange so that $v_p(E) \leq v_p(G)$), and we also name the elements $A$, $B$, $C$, and $D$ as above, so that the Galois conjugate pairs are $\{A,B\}$ and $\{C,D\}$; this means that $\tau$ must act as $(A B)(C D)$.
\end{proof}
\begin{step}\label{Alpha-3}
We have \(N_A = N_B\) and \(N_C = N_D\). 
\end{step}
\begin{proof} \let\qed\relax
This is due to \cref{Judah} since $\tau=(A B)(C D)$ from \cref{Alpha-2}.
\end{proof}
\begin{step} 
We have the following equations:
\begin{align}
\label{Adelie} N_A 			& + N_C = \frac{q-1}{2} \\
\label{Chinstrap} N_A E 		& + N_C G = q \\
\label{Gentoo} N_A E^2 		& + N_C G^2 = q^2 \\
\label{Humboldt} N_A F^2 p 	& + N_C H^2 p = q^2 \\
\label{Emperor} N_A (E^3+ 3 p E F^2) & + N_C (G^3 + 3 p G H^2) = 4 q^2 V^{[1,1]}_1.
\end{align}
\end{step}
\begin{proof}\let\qed\relax
For \(u \in \Ku\), let \(\Omega_u = W_u + W_{-u}\) and \(\Phi_u = W_u - W_{-u}\).
This is consistent with the notation we introduced in \cref{TOR} (with \(k=2\)) and in \cref{PHILIPPE} since \(p \equiv 1 \pmod{2}\) by \cref{Alpha-1}.
Thus, by \cref{Alpha-1}, we have \(\Omega_u = W_u + \tau(W_u)\) and \(\Phi_u = W_u - \tau(W_u)\) for every \(u \in \Ku\).

As we run through \(u \in \Ku\), Steps \ref{Alpha-2} and \ref{Alpha-3} tell us that \(\Omega_u\) has 
\begin{align*}
2 N_A 	& \text{ instances of } E  && 2 N_C \text{ instances of } G
\end{align*}
while \(\Phi_u\) has 
\begin{align*}
N_A		& \text{ instances of } F \sqrt{p}  && N_A \text{ instances of } -F \sqrt{p} \\
N_C		& \text{ instances of } H \sqrt{p}  && N_C \text{ instances of } -H \sqrt{p},
\end{align*}
so that \eqref{Adelie}, \eqref{Chinstrap}, and \eqref{Gentoo} follow from parts \ref{alpaca}, \ref{buffalo}, and \ref{caribou} of \cref{Chlorine} and \eqref{Humboldt} follows from \ref{Calcium}\ref{gnu}. 
The left-hand side of \cref{Silicon}\ref{Lemuel} (with \(t_1 = t_2 = t_3 = 1\)) is summing $W_u^3$ over all $u \in \Ku$, and since $N_A=N_B$ and $N_C=N_D$ by \cref{Alpha-3}, we obtain \eqref{Emperor}.
\end{proof}
\begin{step}\label{Gamma-1}
We have \(G=0\).
\end{step}
\begin{proof}\let\qed\relax
Add \(E G\) times \eqref{Adelie} and \(-(E+G)\) times \eqref{Chinstrap} to \eqref{Gentoo} to get
\begin{equation}\label{Magellanic}
0 = E G \left(\frac{q-1}{2}\right) - (E+G) q + q^2.
\end{equation}
Recall from \cref{Alpha-2} that $v_p(E) \leq v_p(G)$.
Note that if either \(v_p(G) < v_p(q) \) or \(v_p(E) > v_p(q)\), then one of the terms in \eqref{Magellanic} would have strictly lower \(p\)-adic valuation than the other terms. 
Thus, \(v_p(E) \leq v_p(q) \leq v_p(G)\), so that \(E \neq 0\) and \(q \mid G\).
On the other hand, since \(N_A, N_C \in \Z_+\), \eqref{Gentoo} tells us that  
\[
q^2 = N_A E^2 + N_C G^2 > N_C G^2 \geq G^2,
\]
and so \(G=0\).  
\end{proof}
\begin{step}\label{Gamma-2}
We have \(E=q\) and \(N_A = 1\).
\end{step}
\begin{proof}\let\qed\relax
Since \(G=0\) by \cref{Gamma-1}, \eqref{Chinstrap} and \eqref{Gentoo} imply that \(E=q\) and \(N_A = 1\).
\end{proof}
\begin{step}\label{Gamma-3}
We have \(N_C = (q-3)/2\).
\end{step}
\begin{proof}\let\qed\relax
Since \(N_A=1\) by \cref{Gamma-2}, \eqref{Adelie} gives us that \(N_C = (q-3)/2\).
\end{proof}
\begin{step}\label{Gamma-4}
The quantity \(F\) is odd and \(H = 2 I\) for some \(I \in \Z \setminus \{0\}\). 
\end{step}
\begin{proof}\let\qed\relax
We know that \(G=0\) by \cref{Gamma-1} and \(E=q\) by \cref{Gamma-2}.
Furthermore, \(q\) is odd since \(p\) is odd by \cref{Alpha-1}.
\cref{Alpha-2} tells us that \(E\equiv F\pmod{2}\) and \(G\equiv H\pmod{2}\), so \(F\) is odd and \(H\) is even.
But \(H \neq 0\), else \(C=D\) (see \cref{Alpha-2}), so \(H = 2 I\) for some \(I \in \Z \setminus\{0\}\).
\end{proof}
\begin{step}\label{Delta}
We must have \(q \leq 5\).
\end{step}
\begin{proof}\let\qed\relax
We can substitute the results from Steps \ref{Gamma-1}--\ref{Gamma-4} into \eqref{Emperor} and \eqref{Humboldt} to obtain
\begin{align}
\label{Royal} 3 F^2 & = \frac{q}{p} \cdot (4 V^{[1,1]}_1 - q) \\
\label{Rockhopper} F^2 + 2(q-3) I^2 & = \frac{q^2}{p}.
\end{align}
Since \(p \equiv 1 \pmod{4}\) by \cref{Alpha-1}, we have \(\gcd(q/p, 3) = \gcd(q/p, 2(q-3)) = 1\), and therefore since $V^{[1,1]}_1 \in \Z$, \eqref{Royal} and \eqref{Rockhopper} consecutively give us that \((q/p) \mid F^2\) and \((q/p) \mid I^2\). 
Now we know from \cref{Gamma-4} that \(F, I \neq 0\), so \(F^2/(q/p), I^2/(q/p) \geq 1\). 
If we substitute this into \eqref{Rockhopper}, the equality becomes the inequality 
\[
q = \frac{q^2/p}{q/p} = \frac{F^2}{q/p} + 2(q-3) \cdot \frac{I^2}{q/p} \geq 1 + 2(q-3) = 2q - 5,
\]
so that \(q \leq 5\).
\end{proof}
\begin{step}
We conclude that \(q = 5\) and \(s \equiv 3 \pmod{4}\). 
\end{step}
\begin{proof}
Steps \ref{Alpha-1} and \ref{Delta} give us that an action with two disjoint transpositions can only occur  when \(q = p = 5\). 
We now consider the possible values for \(s\). 
Since \(\cW_{K, s'} = \cW_{K, s''}\) if \(s' \equiv s'' \pmod{q-1}\) (see the definition of equivalent exponents in \cref{INTRODUCTION}), it suffices to consider the cases when \(s \equiv 0, 1, 2, 3 \pmod{4}\). 
We cannot have \(s \equiv 0, 2 \pmod{4}\), for then \(\gcd(s, q-1) = \gcd(s,4) \neq 1\), so \(s\) would not be invertible.
Nor can we have  \(s \equiv 1 \pmod{4}\), for then \(s\) would be degenerate and this would make $|\ws|\leq 2$ by \cref{Tor}.
Thus \(s \equiv 3 \pmod{4}\). 
\end{proof}

\begin{proposition}[{\bf No action as a transposition}]\label{Barnabas}
If \(|\ws| = 4\), then \(\tau\) does not permute the elements of \(\ws\) as a transposition.
\end{proposition}
Suppose that \(|\ws| = 4\).
Assume that \(\tau\) permutes \(\ws\) as a transposition to show a contradiction.
For clarity, the proof of this proposition is broken into steps. 
\setcounter{step}{0}
\renewcommand\theHstep{Barnabas.\arabic{step}}
\begin{step}\label{A}
We have \(p \equiv 1 \pmod{4}\), so \(p \geq 5\), and there exist \(A, B, E, F \in \Z\) with \(|A| \leq |B|\), \(F > 0\), and \(E \equiv F \pmod{2}\) such that \(\ws = \{A,B, C = (E+F \sqrt{p})/2, D = (E - F \sqrt{p})/2\}\) and $\tau$ acts on $\ws$ as $(A)(B)(C D)$, i.e., $\tau(A)=A$, $\tau(B)=B$, $\tau(C)=D$, and $\tau(D)=C$.
Moreover, both $N_A$ and $N_B$ are even and \(N_C = N_D\).
\end{step}

\begin{proof}\let\qed\relax
Since \(\tau\) has order \(2\), we know by \cref{Simeon} that \(p \equiv 1 \pmod{4}\) and that \(\Q(\ws) = \Q(\sqrt{p})\), the unique degree 2 extension of \(\Q\) that lies in \(\Qz\). 
This means that the two elements of \(\ws\) exchanged by \(\tau\) are Galois conjugate algebraic integers in \(\Q(\sqrt{p})\), and hence can be written as \(C = (E+F \sqrt{p})/2\) and \(D = (E - F \sqrt{p})/2\) for some \(E,F \in \Z\) where \(E \equiv F \pmod{2}\) and \(F > 0\) by \cref{Benjamin}.
The other two elements of \(\ws\) are algebraic integers fixed by \(\tau\) (and hence by \(\sigma\)), so they must be rational integers; we label these $A$ and $B$ in such a way that $|A| \leq |B|$. 
Then both $N_A$ and $N_B$ are even and \(N_C = N_D\) by \cref{Judah}.
\end{proof}

\begin{step}\label{B}
We have \(\tau(W_u) = W_{-u}\) for all \(u \in \Ku\), so as in Sections \ref{TOR} (with \(k=2\)) and \ref{PHILIPPE} we can let
\begin{align*}
& \Omega = \sum_{u \in \Ku} (W_u + W_{-u}) [u] = \sum_{u \in \Ku} (W_u + \tau(W_u)) [u], \\
& \Phi = \sum_{u \in \Ku} (W_u - W_{-u}) [u] = \sum_{u \in \Ku} (W_u - \tau(W_u)) [u],
\end{align*} 
\[
V = V^{[1,1]}, \qquad U = V^{[1,-1]},  \qquad \text{ and } \qquad T = [1] + [-1].
\]
\end{step}

\begin{proof}\let\qed\relax
Since \cref{A} implies that \(\tau\) has order \(2\) and \(p \equiv 1 \pmod{2}\), \cref{Simeon} and \eqref{Rachel} give us that \(\tau(W_u) = W_{-u}\) for all \(u \in \Ku\) and we have the bilateral symmetry alluded to in Sections \ref{TOR} (with \(k=2\)) and \ref{PHILIPPE}.
\end{proof}

\begin{step}\label{C}
The integer \(E\) is odd and there exist rational integers \(X < Y < Z\) such that \(\{X,Y,Z\} = \{2 A, 2 B, E\}\). 
Let \(M_R = |\{u \in \Ku : \Omega_u = R\}|\) for \(R \in \{X, Y, Z\}\).
Then we have \(\{(X, M_X), (Y, M_Y), (Z, M_Z)\} = \{(2 A, N_A), (2 B, N_B), (E, 2 N_C)\}\) and the following equations hold: 
\begin{align}
\label{Avocet} 	q-1 			&= M_X + M_Y + M_Z \\
\label{Bluebird} 	2 q 			&= M_X X + M_Y Y + M_Z Z \\
\label{Crane} 	2 q ^2 		&= M_X X^2 + M_Y Y^2 + M_Z Z^2 \\
\label{Eagle} 	M_X 			&= \frac{2 q^2 - 2 q(Y+Z) + (q-1)Y Z}{(X-Y)(X-Z)} \\
\label{Flamingo} 	M_Y 			&= \frac{2 q^2 - 2 q(X+Z) + (q-1)X Z}{(Y-X)(Y-Z)} \\
\label{Grosbeak} 	M_Z 			&= \frac{2 q^2 - 2 q(X+Y) + (q-1)X Y}{(Z-X)(Z-Y)} \\
\label{Hawk} 	V_1 + 3 V_{-1} 	&= (X+Y+Z) - \frac{X Y + Y Z + Z X}{q} + \frac{(q-1) X Y Z}{2 q^2} \\
\label{Illadopsis} 	q^2 			&= N_C F^2 p \\
\label{Junglefowl} 	V_1 - V_{-1} 	&= E \\
\label{Kea} 		2 F^2 p	 	&= \sum_{u \in \Ku} (V_u + V_{-u} - 2 U_u)^2.
\end{align}
\end{step}
\begin{proof}\let\qed\relax
Since \(N_C = N_D\) by \cref{A}, we observe that as \(u\) runs through \(\Ku\),
\begin{center}
\begin{tabular}{c|c}
\(\Phi_u\) has & \(\Omega_u\) has \\
\hline
\(N_A + N_B\) instances of \(0\) 	& \(N_A\) instances of \(2 A\) \\
\(N_C\) instances of \(F \sqrt{p}\) & \(N_B\) instances of \(2 B\) \\
\(N_C\) instances of \(- F \sqrt{p}\) & \(2 N_C\) instances of \(E\)
\end{tabular}
\end{center}
and $\Omega_u=E$ for those $u$ such that $\Phi_u\not=0$.
Thus, we obtain \eqref{Illadopsis}--\eqref{Kea} from \cref{Calcium}\ref{gnu}--\ref{ibex}.
Note that \eqref{Illadopsis} and \cref{A} imply that \(E\) and \(F\) are odd, whereas \(2 A\) and \(2 B\) must be distinct and even, so that there are rational integers \(X < Y < Z\) with \(\{X,Y,Z\} = \{2 A, 2 B, E\}\) and we can let \(M_X\), \(M_Y\), and \(M_Z\) be as stated above. 
Equations \eqref{Avocet}--\eqref{Crane} then follow from \cref{Chlorine}\ref{alpaca}--\ref{caribou}, which we also use to prove \eqref{Eagle} from the following observation:
\[
M_X (X-Y)(X-Z) = \sum_{u \in \Ku} (\Omega_u - Y)(\Omega_u - Z),
\]
and \eqref{Flamingo} and \eqref{Grosbeak} follow similarly by exchanging the roles of $X$, $Y$, and $Z$.
Similarly, one can prove \eqref{Hawk} using all parts of \cref{Chlorine} (and the fact that $V=V^{[1,1]}$) from the following observation:
\[
0 = \sum_{u \in \Ku} (\Omega_u - X)(\Omega_u - Y)(\Omega_u - Z).
\]
\end{proof}

\begin{step}\label{D}
We have \(-q < -2(q-1)/(p-1) < X < Y < Z < 2 q\) and \(v_p(X), v_p(Y), v_p(Z) \geq 1\).  If any of \(X\), \(Y\), or \(Z\) is nonzero, then its \(p\)-adic valuation is less than the \(p\)-adic valuation of \(q\).
If none of \(X\), \(Y\), and \(Z\) is zero, then \(v_p(XY), v_p(YZ), v_p(ZX) > v_p(q)\).
\end{step}
\begin{proof} \let\qed\relax
The first chain of inequalities follows from \cref{C} and \cref{Benjamin} (which applies due to \cref{A} and \cref{Tor}), once we notice that \(2 A\), \(2 B\), and \(E\) take the place of \(I\) in \cref{Benjamin}.
\cref{Benjamin} also tells us that $v_p(X), v_p(Y), v_p(Z) \geq 1$.
Next, \(M_X X^2\), \(M_Y Y^2\), and \(M_Z Z^2\) are all even rational integers by \cref{C}, so if \(X \neq 0\) but \(v_p(X) \geq v_p(q)\), then \(2 q^2 \mid M_X X^2\), and hence \(M_X X^2 = 2 q^2\) and \(Y = Z = 0\) by \eqref{Crane}.
This contradicts \cref{C}.
Analogous arguments show that the same result holds for \(Y\) and \(Z\).
In particular, if \(0 \not\in \{X,Y,Z\}\), then \(v_p(X), v_p(Y), v_p(Z) < v_p(q)\).
Thus, if we write \eqref{Hawk} as
\[
2 q^2 (V_1 + 3 V_{-1}) = 2 q^2 (X+Y+Z) - 2 q(X Y + Y Z + Z X) + (q-1)(X Y Z),
\]
then \((q-1) X Y Z\) has a strictly smaller \(p\)-adic valuation than every other term on the right-hand side of the above equation. 
This implies that  
\[
v_p(X) + v_p(Y) + v_p(Z) = v_p(2 q^2 (V_1 + 3 V_{-1})) \geq 2 v_p(q),
\] 
and so the desired inequalities follow from subtracting one of the terms on the left-hand side from both sides. 
\end{proof}

\begin{step}\label{E}
We have \(-q < X < Y = 0 < Z < q\).
\end{step}
\begin{proof} \let\qed\relax
Recall from \cref{C} that $M_X$ is a strictly positive count, so the numerator and denominator in \eqref{Eagle} must have the same sign.
Thus, to prove this step, it suffices to show that \(Y=0\) since \(-q < X < Y < Z\) by  \cref{D}, for then the numerator in \eqref{Eagle}, which is positive, becomes \(2 q (q - Z)\). 

Suppose that \(Y \neq 0\).
By \cref{C}, we know that \(Z > 0\), since otherwise the right-hand side of \eqref{Bluebird} would be negative.
Moreover, the numerator in \eqref{Eagle} is positive, that is, 
\begin{equation}\label{Emu}
2q^2 - 2q(Y+Z) + (q-1)Y Z > 0.
\end{equation}
Thus, using \cref{D} and the fact that \(p \geq 5\) from \cref{A} in \eqref{Emu} gives us
\[
Y Z > \frac{2q (Y+Z-q)}{q-1} > \frac{2q}{q-1} \left(-2\left(\frac{q-1}{5-1}\right) + 1 - q\right) = -3 q > -p q.
\]
We cannot have \(Y<0\), for that would imply both \(0 \not\in \{X,Y,Z\}\) and \(v_p(Y Z) \leq v_p(q)\), which contradicts \cref{D}.  So we must have \(Y>0\).
If we use the same argument, replacing \eqref{Eagle} with \eqref{Grosbeak}, \(Z\) with \(Y\), and \(Y\) with \(X\), we show that \(X <0\) is also impossible, and so obtain \(X \geq 0\). 

If \(X > 0\), then \cref{D} implies that \(X, Y, Z \geq p\), so \eqref{Bluebird}, \eqref{Avocet}, and the fact that \(q \geq p \geq 5\) by \cref{A} give us the contradiction
\begin{align*}
2 q \geq p (M_X + M_Y + M_Z) = p(q-1) \geq 5 q - p \geq  4 q.
\end{align*}
This forces \(X = 0 < Y < Z\) by \cref{C}, so that \eqref{Bluebird} and \eqref{Crane} give 
\[
2 q Z = (M_Y Y + M_Z Z) Z > M_Y Y^2 + M_Z Z^2 = 2 q^2,
\]
and hence \(Z > q\). 
On the other hand, \(M_Z Z^2 \leq 2q^2\) by \eqref{Crane}, so \(M_Z = 1\).
With this information, \eqref{Bluebird} and \eqref{Crane} become 
\begin{align}
\label{Bullfinch} M_Y Y + Z & = 2 q \\
\label{Cormorant} M_Y Y^2 + Z^2 & = 2 q^2.  
\end{align}
Since \(v_p(Z) < v_p(q)\) by \cref{D}, \eqref{Bullfinch} and \eqref{Cormorant} imply that \(v_p(M_Y Y) = v_p(Z)\) and \(v_p(M_Y Y^2) = v_p(Z^2)\), and hence that \(v_p(M_Y) = 0\) and \(v_p(Y) = v_p(Z)\). 
Moreover, \eqref{Flamingo} can be rewritten as 
\[
M_Y Y = \frac{2 q (Z - q)}{Z-Y},
\] 
so that \(v_p(Y) = v_p(q) + v_p(Z) - v_p(Z-Y)\), and so \(v_p(Z-Y) = v_p(q)\). 
In other words, \(q \mid Z-Y\).
Since \(0 < Y < Z < 2 q\) by \cref{D}, this is only possible if \(Z = Y + q\). 
If we substitute this equation for \(Z\) into \eqref{Bullfinch} and \eqref{Cormorant} and solve for \(Y\), we obtain \(3 Y = q\), which is impossible because \(p \equiv 1 \pmod{4}\) by \cref{A}. 
\end{proof}

\begin{step}\label{F}
We have \(E = X < 0\) and \(A = Y/2 = 0\) and \(B = Z/2 > 0\).
Moreover, \(|E| < |B|\) and \(V_{-1} > 0\) and \(B = 2 V_{-1}/(1-E/q)\).
\end{step}
\begin{proof} \let\qed\relax
Recall from \cref{C} that $\{X,Y,Z\}=\{2 A, 2 B, E\}$ is a set of three distinct numbers and that $E$ is odd.
Since \(|A| \leq |B|\) by \cref{A} and \(Y=0\) is even by \cref{E}, we  must have \(0 = Y = A\) and \(\{X,Z\} = \{2 B, E\}\).
We obtain \(B = 2 V_{-1}/(1-E/q)\) by substituting these facts into \eqref{Hawk} and using \eqref{Junglefowl} (note that we can divide by \(1 - E/q\) since \cref{E} implies that \(|E| < q\)).

Now, since \(B\) is nonzero, \(1-E/q\) is positive (since \(|E|< q\)), and \(V_{-1}\) is nonnegative (by \cref{Potassium}\ref{kangaroo}), we must have \(B = Z/2\) is positive, and hence \(V_{-1} > 0\) and \(E = X < 0\). 
It then follows that \(1 < 1 - E/q < 2\) and \(V_{-1} < B < 2 V_{-1}\).
Lastly, \cref{Potassium}\ref{kangaroo} tells us that \(V_1 \geq 0\), so \(E \geq -V_{-1}\) by \eqref{Junglefowl}, and thus \(|E| < |B|\).
\end{proof}

\begin{step} \label{G}
There exists an odd integer \(m\) with \(0 < m < n\) such that \(N_C = p^m\) and \(F = p^{n-(m+1)/2}\). 
Let \(\ell = v_p(B) - v_p(E)\).
Then \(v_p(N_B) = m-2 \ell\).
Moreover, we have 
\begin{align}
\label{Birch} N_B B + N_C E 				& = q \\
\label{Cedar} 2 N_B B^2 + N_C E^2		 	& = q^2.
\end{align}
\end{step}
\begin{proof} \let\qed\relax
The results about \(m\), \(N_C\), and \(F\) follow from \eqref{Illadopsis} since \(N_C < q\), while \eqref{Birch} and \eqref{Cedar} come from equations \eqref{Bluebird} and \eqref{Crane} and Steps \ref{C} and \ref{F}. 
Lastly, \eqref{Cedar} implies that \(v_p(N_B B^2) =  v_p(N_C E^2)\) since \(2 N_B B^2 > 0\) and \(N_C E^2 > 0\) by \cref{F} and \(p \nmid 2\) by \cref{A}, so \(v_p(N_B) = m - 2 \ell\).
\end{proof}

\begin{step}\label{H}
We have \(\ell > 0\), and there exist \(\beta, \epsilon, \nu \in \Z_+\) all relatively prime to \(p\) such that 
\begin{equation}\label{Sunflower}
B = \beta p^{n-m+2 \ell}, \qquad E = - \epsilon p^{n-m+\ell}, \text{ and}  \qquad N_B = 2 \nu p^{m-2 \ell}.
\end{equation}
Moreover, we have the following equations:
\begin{align}
\label{Bluebell} 	2 \nu \beta - \epsilon p^\ell 	& = 1 \\
\label{Crocus}	4 \nu \beta^2 + \epsilon^2 		& = p^{m-2 \ell}.
\end{align}
\end{step} 

\begin{proof}\let\qed\relax
Steps \ref{A}, \ref{F}, and \ref{G} allow us to write \(B = \beta p^{v_p(E) + \ell}\), \(N_B = 2 \nu p^{m-2 \ell}\), \(E = - \epsilon p^{v_p(E)}\), and \(N_C = p^m\)
with \(\beta, \epsilon, \nu \in \Z_+\) all relatively prime to \(p\), so that \eqref{Birch} and \eqref{Cedar} become 
\begin{align}
\label{Bison} 2 \nu \beta p^{m + v_p(E) - \ell} - \epsilon p^{m + v_p(E)} & = p^n \\
\label{Cow} 4 \nu \beta^2 + \epsilon^2 & = p^{2 n - m - 2 v_p(E)}.
\end{align}
It thus suffices to show that \(\ell > 0\), for then \(m + v_p(E) - \ell < m+v_p(E)\), and hence \(m + v_p(E) - \ell = n\) by \eqref{Bison}, so that \(v_p(E) = n - m + \ell\), and so the expressions for $E$ and $B$ at the beginning of this proof become those in \eqref{Sunflower} while \eqref{Bison} and \eqref{Cow} become \eqref{Bluebell} and \eqref{Crocus}.

Suppose \(\ell \leq 0\), and let \(g = n - m - v_p(E)\). 
Then \eqref{Bison} and \eqref{Cow} become 
\begin{equation}\label{Peony}
2 \nu \beta p^{-\ell} - \epsilon = p^g
\end{equation}
\begin{equation}\label{Petunia}
4 \nu \beta^2 + \epsilon^2 = p^{2 g + m}.
\end{equation}
By \cref{F}, we have \(\epsilon = |E|/p^{v_p(E)} < |B|/p^{v_p(E) + \ell} = \beta\), so \eqref{Peony} gives us  
\begin{equation}\label{Chrysanthemum}
p^g > 2 \nu \beta p^{-\ell}- \beta \geq \beta(2 \nu - 1),
\end{equation}
and hence \(g > 0\) since \(\beta, \nu \geq 1\). 
Note that this implies that \(\ell = 0\), for otherwise the \(p\)-adic valuation of the left-hand side of \eqref{Peony} would be \(0\).  
We can thus solve \eqref{Peony} for \(\epsilon\) and substitute the resulting expression into \eqref{Petunia} to get 
\begin{equation}\label{Poppy}
4 \beta^2 \nu (\nu + 1) - 4 \nu \beta p^{g} + p^{2 g} - p^{2 g + m} = 0.
\end{equation}
Since \(g > 0\), the third and fourth terms on the left-hand side of \eqref{Poppy} have strictly larger \(p\)-adic valuation than the second term does, so we must have \(v_p(4 \beta^2 \nu (\nu+1)) = v_p(4 \nu \beta p^g)\), that is, 
\(v_p(\nu+1) = g\).
So \(\nu = -1 + \mu p^g\) for some \(\mu \geq 1\) such that \(p \nmid \mu\).
But if we substitute this into \(p^g > \beta(2\nu-1)\) from \eqref{Chrysanthemum} and rearrange to obtain an upper bound for $\mu$, then (keeping in mind that $p \geq 5$ by \cref{A}) we obtain
\[
\mu < \frac{1}{2 \beta} + \frac{3}{2 p^g} \leq \frac{1}{2} + \frac{3}{2 \cdot 5} = \frac{4}{5},
\]
which is a contradiction. 
We thus have \(\ell > 0\), as we wished. 
\end{proof}

\begin{step}\label{I}
We have both \(B E - 2 C D = \beta p^{2 n- 2 m + 2 \ell}\) and \(C^2 + D^2 - B E = p^{2 n- 2 m + 2 \ell} (p^{m - 2 \ell} - \beta)\).
\end{step}
\begin{proof} \let\qed\relax
These results come from using the expressions for $C$ and $D$ in \cref{A} and those for \(B\), \(E\), and \(F\) in Steps \ref{G} and \ref{H} to write  
\begin{align*}
B E - 2 C D & = p^{2 n - 2 m + 2 \ell} \left(\frac{p^{m-2 \ell}-\epsilon^2}{2}-\beta\epsilon p^\ell \right) \\
C^2 + D^2 - B E & = p^{2 n - 2 m + 2 \ell} \left(\frac{p^{m-2 \ell}+\epsilon^2}{2} + \beta \epsilon p^{\ell}\right)
\end{align*}
and then using \eqref{Bluebell} and \eqref{Crocus} to simplify these expressions.
\end{proof}

\begin{step} 
Let \(S_R = \{u \in \Ku: W_u = R\}\) for \(R \in \ws\).  If we identify these subsets of $\Ku$ with group algebra elements as described before \cref{Hydrogen}, then, using the definitions of $W=\weil$ from \eqref{celery} in \cref{HERBERT} and of $T$, $U$, $V$ from \cref{B}, we have
\begin{align}
\label{Exodus} W T & = \sum_{u \in \Ku} (W_u + W_{-u}) [u] = 2 B S_B + E (S_C + S_D) \\
\label{Leviticus} W U & = \sum_{u \in \Ku} W_u W_{-u} [u] 	= B^2 S_B + C D (S_C + S_D) \\
\label{Genesis} W V T & = \sum_{u \in \Ku} (W_u^2+W_{-u}^2) [u] = 2 B^2 S_B + (C^2 + D^2)(S_C + S_D).
\end{align}
\end{step}

\begin{proof} \let\qed\relax
The left-hand equalities follow from the definitions of $T$, $U$, and $V$ and also \cref{Aluminum} in the case of \eqref{Leviticus} and \eqref{Genesis}.  The right-hand equalities follow from the fact that $W_{-u}=\tau(W_u)$ (by Step \ref{B}) and the values for $W_u$ in Steps \ref{A} and \ref{F}.
\end{proof}

\begin{step}\label{K} 
We have \(\beta = 1\), so \(B = p^{n-m+2 \ell}\).
\end{step}
\begin{proof} \let\qed\relax
We can eliminate \(S_C + S_D\) from \eqref{Exodus} and \eqref{Leviticus} to get 
\begin{equation}\label{Numbers}
W (E U - C D T) = B(B E -2 C D) S_B. 
\end{equation}
Then we can multiply both sides of \eqref{Numbers} by \(\conj{W}/(B(B E - 2 C D))\) to get, by \cref{Oxygen} and Steps \ref{H} and \ref{I}, that 
\begin{equation}\label{Joshua}
\conj{W} S_B =\frac{q^2 (E U - C D T)}{\beta^2 p^{3 n - 3 m + 4 \ell}}.
\end{equation}
Note that the coefficients of \(\conj{W} S_B \in \gr\) are algebraic integers, while the coefficients of the right-hand side of \eqref{Joshua} are rational numbers, so the coefficients in \eqref{Joshua} must all be rational integers. 
In particular, \(\beta\) divides every coefficient of the numerator of the right-hand side of the above equation.
Since \(\gcd(q, \beta) = 1 = \gcd(\beta,\epsilon) = \gcd(\beta,E)\) by \cref{H} and \eqref{Bluebell}, we must have \(\beta \mid U_u\) for all \(u \not\in \{\pm 1\}\).
If \(\beta > 1\), then \(\beta \nmid -1\), so \(U_u \geq 0\) for every \(u \not= \pm 1\) by \cref{Potassium}\ref{kangaroo}.
We also know that \(U_1 = U_{-1} = V_{-1} > 0\) by \cref{Potassium}\ref{wallaby} and \cref{F}. 
But then \(\sum_{u \in \Ku} U_u > 0\), which contradicts \cref{Potassium}\ref{wombat}.
Thus \(\beta=1\) and \(B = p^{n-m+2 \ell}\) by \cref{H}.
\end{proof}

\begin{step}\label{L}
We have \(3 \leq 3 \ell < m < 4 \ell\) and  
\begin{equation}\label{squirrel}
\epsilon^2 + 2 p^\ell \epsilon -(p^{m-2 \ell}-2)=0.
\end{equation}
\end{step}
\begin{proof}\let\qed\relax
\cref{squirrel} comes from using \cref{K} and eliminating \(\nu\) from \eqref{Bluebell} and \eqref{Crocus}. 
Then \eqref{squirrel} and \cref{H} imply that \(p^{m - 3 \ell} > 2 \epsilon > 1\), so \(m > 3 \ell \geq 3\).

Recall from \cref{H} that $\epsilon > 0$, so \eqref{squirrel} implies $\epsilon = -p^\ell + \sqrt{p^{2 \ell}+p^{m-2 \ell}-2}$.
If we assume that \(m > 4 \ell\), then $\epsilon \geq -p^\ell + \sqrt{p^{2 \ell}+p^{2 \ell + 1}-2}$, and since $\ell > 0$ by \cref{H}, we obtain $\epsilon \geq - p^\ell + \sqrt{p} \cdot p^\ell  > p^\ell$ because \(\sqrt{p} \geq \sqrt{5} > 2\) by \cref{A}. 
But then Steps \ref{H} and \ref{K} give us that \(|E| = \epsilon p^{n-m+\ell} > p^{n-m+2\ell} = |B|\), which contradicts \cref{F}.
So $m \leq 4\ell$, and this inequality is actually strict since $m$ is odd by \cref{G}.
\end{proof}

\begin{step}\label{M}
Let \(\delta_0 = 1\) and \(\delta_x = 0\) if \(x \not= 0\). 
For any \(u \in \Ku\), we have   
\[
V_u + V_{-u} = q \delta_{u^2-1} - 2 (p^{m - 2 \ell} - 1) U_u.
\]
Moreover, if \(u \not\in \{\pm 1\}\), then \(U_u \in \{-1,0\}\).
\end{step}
\begin{proof}\let\qed\relax
First, we eliminate \(S_B\) from \eqref{Exodus} and \eqref{Leviticus} (respectively,  \eqref{Exodus} and \eqref{Genesis}) and use Steps \ref{I} and \ref{K} to get  
\begin{align} 
\label{Flowerpot} W (B T - 2 U) & = p^{2 n - 2 m + 2 \ell} (S_C + S_D) \\
\label{Garden} W (V T - B T) & = (p^{m - 2 \ell} - 1) p^{2 n - 2 m + 2 \ell} (S_C + S_D).
\end{align}
Then, we substitute \eqref{Flowerpot} into \eqref{Garden} and use \cref{K} to obtain 
\[
W \left(-2(p^{m-2 \ell}-1)U - V T +  q T\right) = 0.
\]
Note that \(W\) is a unit in \(\gr\) because \(W \conj{W} = q^2\) by \cref{Oxygen}, so 
\[ V T = q T - 2(p^{m-2 \ell}-1) U, \]
and hence we achieve the above general result. 
When \(u \not\in \{\pm 1\}\), we also have \(U_u \in \{-1, 0\}\) since both \(V_u\) and \(V_{-u}\) are nonnegative and \(U_u \geq -1\) by \cref{Potassium}\ref{kangaroo}, while $-2(p^{m-2\ell}-1)$ is strictly negative by \cref{L}.
\end{proof}

\begin{step}
We conclude that \(\tau\) does not permute \(\ws\) as a transposition.
\end{step}
\begin{proof}
Using the expression for \(F\) in \cref{G} and the expression for \(V_u + V_{-u}\) from \cref{M}, \eqref{Kea} becomes 
\begin{equation}\label{John}
2 p^{2 n - m} = \sum_{u \in \{\pm 1\}} (q^2 - 4 q p^{m - 2 \ell} U_u) + 4 p^{2 m - 4 \ell} \sum_{u \in \Ku} U_u^2.
\end{equation}
We now use the fact from \cref{M} that \(U_u \in \{-1,0\}\) for \(u \notin \{\pm 1\}\) to write $\sum_{u \in \Ku} U_u^2 = \sum_{u \in \{\pm 1\}} U_u (U_u+1) - \sum_{u \in \Ku} U_u$.
Then, since \(\sum_{u \in \Ku} U_u = 0\) and \(U_1 = U_{-1} = V_{-1}\) by \cref{Potassium}\ref{wallaby},\ref{wombat}, we see that \eqref{John} simplifies to   
\begin{equation}\label{Jude}
p^{2 n - m} = q^2 - 4 q p^{m-2 \ell} V_{-1} + 4 p^{2 m - 4 \ell} V_{-1} (V_{-1} + 1).
\end{equation}
We also have from \cref{F} that \(2 V_{-1} = B(1-E/q)\), so \(2 p^{m-2 \ell} V_{-1} = q-E\) by \cref{K}.
Using this fact, the expression for \(E\) in \cref{H}, and \eqref{squirrel} in \eqref{Jude}, we obtain
\begin{align*}
p^{2 n-m}
& = q^2 - 2 q (q-E) + 2 p^{m - 2 \ell} (q-E) + (q-E)^2 \\
& = \epsilon^2 p^{2 n-2 m+2 \ell} + 2 \epsilon p^{n-\ell} + 2 p^{n+m -2 \ell} \\
& = p^{2 n-m} -2 \epsilon p^{2 n - 2 m + 3 \ell} - 2 p^{2 n - 2 m + 2 \ell} + 2 \epsilon p^{n - \ell} + 2 p^{n+m-2 \ell},
\end{align*}
that is, 
\begin{equation}\label{Revelation}
p^{2 n - 2 m + 2 \ell} (1+\epsilon p^\ell)  = p^{n-\ell} (\epsilon+  p^{m- \ell}).
\end{equation}
Since \(\ell > 0\) and \(m > \ell\) by \cref{L} and since $p\nmid \epsilon$ (see \cref{H}), \(p\)-adic valuation shows that \(2 n-2 m+2\ell = n-\ell\).  Then divide \eqref{Revelation} by \(p^{n-\ell} = p^{2 n - 2 m + 2 \ell}\) and rearrange to obtain 
\[
\epsilon (p^\ell-1)  = p^{m-\ell} - 1.
\]
This means that \(p^\ell-1 \mid p^{m-\ell}-1\), which implies that \(\ell \mid m-\ell\), and so \(\ell \mid m\).
But \(3 \ell < m < 4 \ell\) by \cref{L}, so we have a contradiction.
\end{proof}
Now we are ready to prove \cref{MainTheorem} (which was restated at the beginning of this section as \cref{Juvenal}).
\renewcommand*{\proofname}{Proof of \cref{MainTheorem}}
\begin{proof}
Suppose \(|\ws|=4\). 
Since \(\tau\) permutes the elements of \(\ws\) (see \eqref{Rachel}), \(\tau\) must act trivially, as a transposition (while keeping two values fixed), as a composition of two disjoint transpositions, as a \(3\)-cycle (while keeping one value fixed), or as a \(4\)-cycle on \(\ws\).
But Propositions \ref{Josiah}, \ref{Paul}, and \ref{Barnabas} exclude the possibilities that \(\tau\) acts as a \(4\)-cycle, as a \(3\)-cycle, and as a transposition, respectively, while \cref{Enoch} states that \(\tau\) permutes \(\ws\) as a composition of two disjoint transpositions precisely when \(q = 5\) and \(s \equiv 3 \pmod{4}\), in which case \(\ws = \{(5 \pm \sqrt{5})/2, \pm \sqrt{5}\}\).
That is, other than the aforementioned case, \(\tau\) can only act trivially on \(\ws\), and hence \(\ws\) is rational by \cref{Simeon}, as we wished to prove. 
\end{proof}

\section*{Acknowledgements}
The authors thank Jason Lo and an anonymous reviewer for helpful suggestions.


\begin{thebibliography}{DFHR06}

\bibitem[AKL15]{Aubry-Katz-Langevin}
Yves Aubry, Daniel~J. Katz, and Philippe Langevin.
\newblock Cyclotomy of {W}eil sums of binomials.
\newblock {\em J. Number Theory}, 154:160--178, 2015.

\bibitem[Aku65]{Akulinichev}
N.~M. Akulini\v{c}ev.
\newblock Bounds for rational trigonometric sums of a special type.
\newblock {\em Dokl. Akad. Nauk SSSR}, 161:743--745, 1965.

\bibitem[Car78]{Carlitz-1978}
L.~Carlitz.
\newblock A note on exponential sums.
\newblock {\em Math. Scand.}, 42(1):39--48, 1978.

\bibitem[Car79]{Carlitz-1979}
L.~Carlitz.
\newblock Explicit evaluation of certain exponential sums.
\newblock {\em Math. Scand.}, 44(1):5--16, 1979.

\bibitem[Cou98]{Coulter}
Robert~S. Coulter.
\newblock Further evaluations of {W}eil sums.
\newblock {\em Acta Arith.}, 86(3):217--226, 1998.

\bibitem[CP03]{Cochrane-Pinner-2003}
Todd Cochrane and Christopher Pinner.
\newblock Stepanov's method applied to binomial exponential sums.
\newblock {\em Q. J. Math.}, 54(3):243--255, 2003.

\bibitem[CP11]{Cochrane-Pinner-2011}
Todd Cochrane and Christopher Pinner.
\newblock Explicit bounds on monomial and binomial exponential sums.
\newblock {\em Q. J. Math.}, 62(2):323--349, 2011.

\bibitem[DFHR06]{Dobbertin-Felke-Helleseth-Rosendahl}
Hans Dobbertin, Patrick Felke, Tor Helleseth, and Petri Rosendahl.
\newblock Niho type cross-correlation functions via {D}ickson polynomials and
  {K}loosterman sums.
\newblock {\em IEEE Trans. Inform. Theory}, 52(2):613--627, 2006.

\bibitem[DH36]{Davenport-Heilbronn}
H.~Davenport and H.~Heilbronn.
\newblock On an exponential sum.
\newblock {\em Proc. London Math. Soc. (2)}, 41(6):449--453, 1936.

\bibitem[Dob98]{Dobbertin}
Hans Dobbertin.
\newblock One-to-one highly nonlinear power functions on {${\rm GF}(2^n)$}.
\newblock {\em Appl. Algebra Engrg. Comm. Comput.}, 9(2):139--152, 1998.

\bibitem[Fen12]{Feng}
Tao Feng.
\newblock On cyclic codes of length {$2^{2^r}-1$} with two zeros whose dual
  codes have three weights.
\newblock {\em Des. Codes Cryptogr.}, 62(3):253--258, 2012.

\bibitem[Hel76]{Helleseth}
Tor Helleseth.
\newblock Some results about the cross-correlation function between two maximal
  linear sequences.
\newblock {\em Discrete Math.}, 16(3):209--232, 1976.

\bibitem[HR05]{Helleseth-Rosendahl}
Tor Helleseth and Petri Rosendahl.
\newblock New pairs of {$m$}-sequences with 4-level cross-correlation.
\newblock {\em Finite Fields Appl.}, 11(4):674--683, 2005.

\bibitem[Kar67]{Karatsuba}
A.~A. Karatsuba.
\newblock Estimates of complete trigonometric sums.
\newblock {\em Mat. Zametki}, 1(2):199--208, 1967.

\bibitem[Kat12]{Katz-2012}
Daniel~J. Katz.
\newblock Weil sums of binomials, three-level cross-correlation, and a
  conjecture of {H}elleseth.
\newblock {\em J. Combin. Theory Ser. A}, 119(8):1644--1659, 2012.

\bibitem[Kat15]{Katz-2015}
Daniel~J. Katz.
\newblock Divisibility of {W}eil sums of binomials.
\newblock {\em Proc. Amer. Math. Soc.}, 143(11):4623--4632, 2015.

\bibitem[Kat19]{Katz-2019}
Daniel~J. Katz.
\newblock Weil sums of binomials: properties, applications and open problems.
\newblock In Kai-Uwe Schmidt and Arne Winterhof, editors, {\em Combinatorics
  and Finite Fields: Difference Sets, Polynomials, Pseudorandomness and
  Applications}, volume~23 of {\em Radon Ser. Comput. Appl. Math.}, pages
  109--134. De Gruyter, Berlin, Boston, 2019.

\bibitem[KL89]{Katz-Livne}
Nicholas Katz and Ron Livn\'e.
\newblock Sommes de {K}loosterman et courbes elliptiques universelles en
  caract\'eristiques {$2$} et {$3$}.
\newblock {\em C. R. Acad. Sci. Paris S\'er. I Math.}, 309(11):723--726, 1989.

\bibitem[KL16]{Katz-Langevin}
Daniel~J. Katz and Philippe Langevin.
\newblock New open problems related to old conjectures by {H}elleseth.
\newblock {\em Cryptogr. Commun.}, 8(2):175--189, 2016.

\bibitem[Klo27]{Kloosterman}
H.~D. Kloosterman.
\newblock On the representation of numbers in the form {$ax^2+by^2+cz^2+dt^2$}.
\newblock {\em Acta Math.}, 49(3-4):407--464, 1927.

\bibitem[Lan90]{Lang}
Serge Lang.
\newblock {\em Cyclotomic fields {I} and {II}}, volume 121 of {\em Graduate
  Texts in Mathematics}.
\newblock Springer-Verlag, New York, second edition, 1990.

\bibitem[Lan02]{Lang-Algebra}
Serge Lang.
\newblock {\em Algebra}, volume 211 of {\em Graduate Texts in Mathematics}.
\newblock Springer-Verlag, New York, third edition, 2002.

\bibitem[LN97]{Lidl-Niederreiter}
Rudolf Lidl and Harald Niederreiter.
\newblock {\em Finite fields}, volume~20 of {\em Encyclopedia of Mathematics
  and its Applications}.
\newblock Cambridge University Press, Cambridge, second edition, 1997.

\bibitem[LW87]{Lachaud-Wolfmann}
Gilles Lachaud and Jacques Wolfmann.
\newblock Sommes de {K}loosterman, courbes elliptiques et codes cycliques en
  caract\'eristique {$2$}.
\newblock {\em C. R. Acad. Sci. Paris S\'er. I Math.}, 305(20):881--883, 1987.

\bibitem[Nih72]{Niho}
Yoji Niho.
\newblock {\em Multi-valued cross-correlation function between two maximal
  linear recursive sequences}.
\newblock PhD thesis, University of Southern California, Los Angeles, 1972.

\bibitem[SV20]{Shparlinski-Voloch}
Igor~E. Shparlinski and Jos\'{e}~Felipe Voloch.
\newblock Binomial exponential sums.
\newblock {\em Ann. Sc. Norm. Super. Pisa Cl. Sci. (5)}, 21:931--941, 2020.

\bibitem[Tra70]{Trachtenberg}
Herbert~Mitchell Trachtenberg.
\newblock {\em On the cross-correlation functions of maximal linear sequences}.
\newblock PhD thesis, University of Southern California, Los Angeles, 1970.

\bibitem[XHW14]{Xia-Helleseth-Wu}
Yongbo Xia, Tor Helleseth, and Gaofei Wu.
\newblock A note on cross-correlation distribution between a ternary
  {$m$}-sequence and its decimated sequence.
\newblock In {\em Sequences and their applications---{SETA} 2014}, volume 8865
  of {\em Lecture Notes in Comput. Sci.}, pages 249--259. Springer, Cham, 2014.

\bibitem[ZLFG14]{Zhang-Li-Feng-Ge}
Tao Zhang, Shuxing Li, Tao Feng, and Gennian Ge.
\newblock Some new results on the cross correlation of {$m$}-sequences.
\newblock {\em IEEE Trans. Inform. Theory}, 60(5):3062--3068, 2014.

\end{thebibliography}
\end{document}